\documentclass[a4paper,10pt]{amsart} %Submitted on February 24, 2024, to Journal of Topology %Revision completed November 24 and submitted on January 30, 2025
\usepackage[utf8]{inputenc}
\usepackage{url}

\usepackage[expansion=false]{microtype}
\usepackage{amssymb,amsmath,amsopn,amsxtra,amsthm,amsfonts,mathtools}
\usepackage{xr}
\usepackage[mathcal]{euscript}
\usepackage{mathalfa}
\usepackage{txfonts}
\usepackage{todonotes}
\usepackage{xy}
\input xy
\xyoption{all}
\usepackage{tikz}
\usetikzlibrary{matrix,arrows,cd}

\usepackage[pdfusetitle,unicode]{hyperref}

\numberwithin{equation}{section}
\makeatletter
\let\c@equation\c@subsection

\makeatother

\newtheorem{cor}[subsection]{Corollary}
\newtheorem{lem}[subsection]{Lemma}

\newtheorem{thm}[subsection]{Theorem}

\newtheorem*{thm*}{Theorem}
\newtheorem*{lem*}{Lemma}

\theoremstyle{definition}
\newtheorem{defn}[subsection]{Definition}
\newtheorem{rem}[subsection]{Remark}

\theoremstyle{remark}

\renewcommand{\eqref}[1]{(\ref{#1})}

\usepackage{tikz}
\usetikzlibrary{matrix}
\usepackage{tikz-cd}
\tikzset{shorten <>/.style={shorten >=#1,shorten <=#1}}

\usepackage{fixltx2e}

\usepackage{xspace}

\usepackage{xifthen}

\usepackage{xparse}

\newcommand{\nc}{\newcommand}
\nc{\renc}{\renewcommand}
\nc{\ssec}{\subsection}
\nc{\sssec}{\subsubsection}
\nc{\on}{\operatorname}
\nc{\term}[1]{#1\xspace}

\DeclareMathSymbol{A}{\mathalpha}{operators}{`A}
\DeclareMathSymbol{B}{\mathalpha}{operators}{`B}
\DeclareMathSymbol{C}{\mathalpha}{operators}{`C}
\DeclareMathSymbol{D}{\mathalpha}{operators}{`D}
\DeclareMathSymbol{E}{\mathalpha}{operators}{`E}
\DeclareMathSymbol{F}{\mathalpha}{operators}{`F}
\DeclareMathSymbol{G}{\mathalpha}{operators}{`G}
\DeclareMathSymbol{H}{\mathalpha}{operators}{`H}
\DeclareMathSymbol{I}{\mathalpha}{operators}{`I}
\DeclareMathSymbol{J}{\mathalpha}{operators}{`J}
\DeclareMathSymbol{K}{\mathalpha}{operators}{`K}
\DeclareMathSymbol{L}{\mathalpha}{operators}{`L}
\DeclareMathSymbol{M}{\mathalpha}{operators}{`M}
\DeclareMathSymbol{N}{\mathalpha}{operators}{`N}
\DeclareMathSymbol{O}{\mathalpha}{operators}{`O}
\DeclareMathSymbol{P}{\mathalpha}{operators}{`P}
\DeclareMathSymbol{Q}{\mathalpha}{operators}{`Q}
\DeclareMathSymbol{R}{\mathalpha}{operators}{`R}
\DeclareMathSymbol{S}{\mathalpha}{operators}{`S}
\DeclareMathSymbol{T}{\mathalpha}{operators}{`T}
\DeclareMathSymbol{U}{\mathalpha}{operators}{`U}
\DeclareMathSymbol{V}{\mathalpha}{operators}{`V}
\DeclareMathSymbol{W}{\mathalpha}{operators}{`W}
\DeclareMathSymbol{X}{\mathalpha}{operators}{`X}
\DeclareMathSymbol{Y}{\mathalpha}{operators}{`Y}
\DeclareMathSymbol{Z}{\mathalpha}{operators}{`Z}

\nc{\sA}{\ensuremath{\mathcal{A}}\xspace}
\nc{\sB}{\ensuremath{\mathcal{B}}\xspace}
\nc{\sC}{\ensuremath{\mathcal{C}}\xspace}
\nc{\sD}{\ensuremath{\mathcal{D}}\xspace}
\nc{\sE}{\ensuremath{\mathcal{E}}\xspace}
\nc{\sF}{\ensuremath{\mathcal{F}}\xspace}
\nc{\sG}{\ensuremath{\mathcal{G}}\xspace}
\nc{\sH}{\ensuremath{\mathcal{H}}\xspace}
\nc{\sI}{\ensuremath{\mathcal{I}}\xspace}
\nc{\sJ}{\ensuremath{\mathcal{J}}\xspace}
\nc{\sK}{\ensuremath{\mathcal{K}}\xspace}
\nc{\sL}{\ensuremath{\mathcal{L}}\xspace}
\nc{\sM}{\ensuremath{\mathcal{M}}\xspace}
\nc{\sN}{\ensuremath{\mathcal{N}}\xspace}
\nc{\sO}{\ensuremath{\mathcal{O}}\xspace}
\nc{\sP}{\ensuremath{\mathcal{P}}\xspace}
\nc{\sQ}{\ensuremath{\mathcal{Q}}\xspace}
\nc{\sR}{\ensuremath{\mathcal{R}}\xspace}
\nc{\sS}{\ensuremath{\mathcal{S}}\xspace}
\nc{\sT}{\ensuremath{\mathcal{T}}\xspace}
\nc{\sU}{\ensuremath{\mathcal{U}}\xspace}
\nc{\sV}{\ensuremath{\mathcal{V}}\xspace}
\nc{\sW}{\ensuremath{\mathcal{W}}\xspace}
\nc{\sX}{\ensuremath{\mathcal{X}}\xspace}
\nc{\sY}{\ensuremath{\mathcal{Y}}\xspace}
\nc{\sZ}{\ensuremath{\mathcal{Z}}\xspace}

\nc{\bA}{\ensuremath{\mathbf{A}}\xspace}
\nc{\bB}{\ensuremath{\mathbf{B}}\xspace}
\nc{\bC}{\ensuremath{\mathbf{C}}\xspace}
\nc{\bD}{\ensuremath{\mathbf{D}}\xspace}
\nc{\bE}{\ensuremath{\mathbf{E}}\xspace}
\nc{\bF}{\ensuremath{\mathbf{F}}\xspace}
\nc{\bG}{\ensuremath{\mathbf{G}}\xspace}
\nc{\bH}{\ensuremath{\mathbf{H}}\xspace}
\nc{\bI}{\ensuremath{\mathbf{I}}\xspace}
\nc{\bJ}{\ensuremath{\mathbf{J}}\xspace}
\nc{\bK}{\ensuremath{\mathbf{K}}\xspace}
\nc{\bL}{\ensuremath{\mathbf{L}}\xspace}
\nc{\bM}{\ensuremath{\mathbf{M}}\xspace}
\nc{\bN}{\ensuremath{\mathbf{N}}\xspace}
\nc{\bO}{\ensuremath{\mathbf{O}}\xspace}
\nc{\bP}{\ensuremath{\mathbf{P}}\xspace}
\nc{\bQ}{\ensuremath{\mathbf{Q}}\xspace}
\nc{\bR}{\ensuremath{\mathbf{R}}\xspace}
\nc{\bS}{\ensuremath{\mathbf{S}}\xspace}
\nc{\bT}{\ensuremath{\mathbf{T}}\xspace}
\nc{\bU}{\ensuremath{\mathbf{U}}\xspace}
\nc{\bV}{\ensuremath{\mathbf{V}}\xspace}
\nc{\bW}{\ensuremath{\mathbf{W}}\xspace}
\nc{\bX}{\ensuremath{\mathbf{X}}\xspace}
\nc{\bY}{\ensuremath{\mathbf{Y}}\xspace}
\nc{\bZ}{\ensuremath{\mathbf{Z}}\xspace}

\nc{\dA}{\ensuremath{\mathds{A}}\xspace}
\nc{\dB}{\ensuremath{\mathds{B}}\xspace}
\nc{\dC}{\ensuremath{\mathds{C}}\xspace}
\nc{\dD}{\ensuremath{\mathds{D}}\xspace}
\nc{\dE}{\ensuremath{\mathds{E}}\xspace}
\nc{\dF}{\ensuremath{\mathds{F}}\xspace}
\nc{\dG}{\ensuremath{\mathds{G}}\xspace}
\nc{\dH}{\ensuremath{\mathds{H}}\xspace}
\nc{\dI}{\ensuremath{\mathds{I}}\xspace}
\nc{\dJ}{\ensuremath{\mathds{J}}\xspace}
\nc{\dK}{\ensuremath{\mathds{K}}\xspace}
\nc{\dL}{\ensuremath{\mathds{L}}\xspace}
\nc{\dM}{\ensuremath{\mathds{M}}\xspace}
\nc{\dN}{\ensuremath{\mathds{N}}\xspace}
\nc{\dO}{\ensuremath{\mathds{O}}\xspace}
\nc{\dP}{\ensuremath{\mathds{P}}\xspace}
\nc{\dQ}{\ensuremath{\mathds{Q}}\xspace}
\nc{\dR}{\ensuremath{\mathds{R}}\xspace}
\nc{\dS}{\ensuremath{\mathds{S}}\xspace}
\nc{\dT}{\ensuremath{\mathds{T}}\xspace}
\nc{\dU}{\ensuremath{\mathds{U}}\xspace}
\nc{\dV}{\ensuremath{\mathds{V}}\xspace}
\nc{\dW}{\ensuremath{\mathds{W}}\xspace}
\nc{\dX}{\ensuremath{\mathds{X}}\xspace}
\nc{\dY}{\ensuremath{\mathds{Y}}\xspace}
\nc{\dZ}{\ensuremath{\mathds{Z}}\xspace}

\nc{\bbA}{\ensuremath{\mathbb{A}}\xspace}
\nc{\bbB}{\ensuremath{\mathbb{B}}\xspace}
\nc{\bbC}{\ensuremath{\mathbb{C}}\xspace}
\nc{\bbD}{\ensuremath{\mathbb{D}}\xspace}
\nc{\bbE}{\ensuremath{\mathbb{E}}\xspace}
\nc{\bbF}{\ensuremath{\mathbb{F}}\xspace}
\nc{\bbG}{\ensuremath{\mathbb{G}}\xspace}
\nc{\bbH}{\ensuremath{\mathbb{H}}\xspace}
\nc{\bbI}{\ensuremath{\mathbb{I}}\xspace}
\nc{\bbJ}{\ensuremath{\mathbb{J}}\xspace}
\nc{\bbK}{\ensuremath{\mathbb{K}}\xspace}
\nc{\bbL}{\ensuremath{\mathbb{L}}\xspace}
\nc{\bbM}{\ensuremath{\mathbb{M}}\xspace}
\nc{\bbN}{\ensuremath{\mathbb{N}}\xspace}
\nc{\bbO}{\ensuremath{\mathbb{O}}\xspace}
\nc{\bbP}{\ensuremath{\mathbb{P}}\xspace}
\nc{\bbQ}{\ensuremath{\mathbb{Q}}\xspace}
\nc{\bbR}{\ensuremath{\mathbb{R}}\xspace}
\nc{\bbS}{\ensuremath{\mathbb{S}}\xspace}
\nc{\bbT}{\ensuremath{\mathbb{T}}\xspace}
\nc{\bbU}{\ensuremath{\mathbb{U}}\xspace}
\nc{\bbV}{\ensuremath{\mathbb{V}}\xspace}
\nc{\bbW}{\ensuremath{\mathbb{W}}\xspace}
\nc{\bbX}{\ensuremath{\mathbb{X}}\xspace}
\nc{\bbY}{\ensuremath{\mathbb{Y}}\xspace}
\nc{\bbZ}{\ensuremath{\mathbb{Z}}\xspace}

\nc{\mrm}[1]{\ensuremath{\mathrm{#1}}\xspace}
\nc{\mbf}[1]{\ensuremath{\mathbf{#1}}\xspace}
\nc{\mcal}[1]{\ensuremath{\mathcal{#1}}\xspace}
\nc{\msc}[1]{\ensuremath{\mathscr{#1}}\xspace}

\renc{\bar}[1]{\overline{#1}}

\nc{\sub}{\subset}
\nc{\too}{\longrightarrow}
\nc{\hook}{\hookrightarrow}
\nc*{\hooklongrightarrow}{\ensuremath{\lhook\joinrel\relbar\joinrel\rightarrow}}
\nc{\hooklong}{\hooklongrightarrow}
\nc{\twoheadlongrightarrow}{\relbar\joinrel\twoheadrightarrow}
\nc{\shiso}{\approx}
\nc{\isoto}{\xrightarrow{\sim}}
\nc{\isofrom}{\xleftarrow{\sim}}
\renc{\ge}{\geqslant}
\renc{\le}{\leqslant}
\renc{\geq}{\geqslant}
\renc{\leq}{\leqslant}

\nc{\id}{\mathrm{id}}
\nc{\Id}{\mathrm{Id}}
\nc{\can}{\mathrm{can}}

\let\Im\relax
\DeclareMathOperator{\Im}{\mathrm{Im}}
\DeclareMathOperator{\rk}{\mathrm{rk}}
\DeclareMathOperator{\sgn}{\mathrm{sgn}}
\DeclareMathOperator{\Hom}{\mathrm{Hom}}
\DeclareMathOperator{\Real}{\mathrm{Re}}
\nc{\uHom}{\underline{\smash{\Hom}}}

\DeclareMathOperator{\End}{\mathrm{End}}

\nc{\Pre}{\mathrm{PSh}{}}
\nc{\uEnd}{\underline{\smash{\End}}}
\DeclareMathOperator{\codim}{\mathrm{codim}}
\renc{\lim}{\operatorname*{lim}}
\nc{\colim}{\operatorname*{colim}}
\renc{\coprod}{\sqcup}

\nc{\bDelta}{\mbf{\Delta}}
\nc{\DM}{\mbf{DM}}
\nc{\eff}{\mathrm{eff}}
\nc{\bir}{\mathrm{bir}}
\nc{\SB}{\mathrm{SB}}
\nc{\veff}{\mathrm{veff}}
\nc{\cyc}{{\mrm{cyc}}}
\nc{\Cor}{{\mrm{Cor}}}
\nc{\corr}{{\on{corr}}}
\nc{\fet}{{\mrm{f\acute et}}}
\nc{\fsyn}{{\mrm{fsyn}}}
\nc{\fflat}{{\mrm{fflat}}}
\nc{\syn}{{\mrm{syn}}}
\nc{\Br}{{\mrm{Br}}}
\nc{\Perf}{\mathcal{P}\mrm{erf}}
\nc{\QCoh}{\mathcal{Q}\mrm{Coh}}
\nc{\Az}{\mathcal{A}\mrm{z}}
\nc{\Pic}{\mrm{Pic}}
\nc{\Nrd}{\mrm{Nrd}}
\nc{\perf}{\mrm{perf}}
\nc{\oblv}{\on{oblv}}
\nc{\exact}{\on{exact}}

\nc{\F}{{\on{F}}}
\nc{\clopen}{{\mrm{clopen}}}
\nc{\B}{\mrm{B}}
\nc{\D}{\mrm{D}}
\nc{\Fin}{\on{Fin}}
\nc{\Cut}{\on{Cut}}
\nc{\Cart}{\on{Cart}}
\nc{\pairs}{\mathsf{pairs}}
\nc{\Pairs}{\mathrm{Pair}}
\nc{\Trip}{\mathrm{Trip}}
\nc{\Lab}{\mathrm{Lab}}
\nc{\coCart}{\mathrm{coCart}}
\nc{\RKE}{\mathrm{RKE}}
\nc{\strict}{\mathrm{strict}}
\nc{\Emb}{\mathrm{Emb}}
\nc{\EMB}{\mathcal{E}\mathrm{mb}}
\nc{\Split}{\mathrm{Split}}
\nc{\Set}{\mathrm{Set}}
\nc{\sSets}{\mathrm{sSets}}
\nc{\pb}{\mathrm{pb}}
\nc{\lci}{\mathrm{lci}}
\nc{\fib}{\mathrm{fib}}
\nc{\cofib}{\mathrm{cofib}}
\nc{\diff}{\mrm{diff}}
\nc{\gp}{\mrm{gp}}
\nc{\ind}{\mrm{ind}}
\nc{\chr}{\mrm{char}}
\nc{\mgp}{\mrm{mot-gp}}
\nc{\FSyn}{\mrm{FSyn}}
\nc{\FFlat}{{\mrm{FFlat}}}
\nc{\FEt}{\mrm{FEt}}
\nc{\Spc}{\mrm{Spc}}
\nc{\Ob}{\mrm{Ob}}
\nc{\Spt}{\mrm{Spt}}
\nc{\tors}{\mrm{tors}}
\nc{\T}{\mathrm{T}}
\nc{\suspinf}{\Sigma^\infty}
\nc{\h}{\mrm{h}}
\nc{\uhom}{\underline{\mathrm{Hom}}}
\nc{\umap}{\underline{\mathrm{Maps}}}
\nc{\Map}{\mathrm{Map}}
\nc{\map}{\operatorname{map}}           % mapping spectrum
\nc{\Autom}{\mathrm{Aut}}
\renc{\H}{\bH}
\nc{\Einfty}{{\sE_\infty}}
\nc{\Eone}{{\sE_1}}
\nc{\Stab}{\mrm{Stab}}
\nc{\lax}{{\mrm{lax}}}
\nc{\cocart}{{\mrm{cocart}}}
\nc{\Sch}{\mbf{Sch}}
\nc{\dSch}{\mrm{dSch}}
\nc{\Aff}{\mrm{Aff}}
\nc{\SmAff}{\mrm{SmAff}}
\nc{\dAff}{\mrm{dAff}}
\nc{\Fr}{\on{Fr}}
\nc{\A}{\mathbf{A}}
\nc{\Pp}{\mathbf{P}}
\nc{\Ss}{\mathbf{S}}
\nc{\N}{\mathbb{N}}
\nc{\Z}{\mathbb{Z}}
\nc{\Q}{\mathbb{Q}}
\nc{\Oo}{\mathcal{O}} 
\nc{\Aa}{\mathcal{A}} 
\nc{\Bb}{\mathcal{B}} 
\nc{\Ee}{\mathcal{E}}
\nc{\Ff}{\mathcal{F}} 
\nc{\Gg}{\mathcal{G}} 
\nc{\Vv}{\mathcal{V}}
\nc{\red}{{\on{red}}}
\nc{\Voev}{{\on{Voev}}}
\nc{\Corr}{\mrm{Corr}}
\nc{\Span}{\mathbf{Corr}}
\nc{\Gap}{\mrm{Gap}}
\nc{\Filt}{\mrm{Filt}}
\nc{\Corrfr}{\Corr^{\fr}}
\nc{\Corrvfr}{\Corr^{\Vfr}}
\nc{\Spec}{\on{Spec}}
\nc{\Sm}{\mbf{Sm}}
\nc{\QSm}{\mrm{QSm}}
\nc{\Gm}{\mathbb{G}_{\mrm{m}}}
\renc{\P}{\bP}
\nc{\nis}{\mathrm{nis}}
\nc{\KH}{\mathrm{KH}}
\nc{\bil}{\mathrm{bil}}
\nc{\Zar}{\mathrm{Zar}}
\nc{\zar}{\mathrm{zar}}
\nc{\Nis}{\mathrm{Nis}}
\nc{\et}{\mathrm{\acute et}}
\nc{\all}{\mathrm{all}}
\nc{\fold}{\mathrm{fold}}
\nc{\Fun}{\mathrm{Fun}}
\nc{\Ho}{\mathrm{Ho}}
\nc{\Segal}{\mathrm{Segal}}
\nc{\Mon}{\mrm{Mon}{}}
\nc{\Ab}{\mrm{Ab}}
\nc{\Gr}{\mrm{Gr}}
\nc{\Sh}{\on{Sh}}
\nc{\M}{\mrm{M}}
\nc{\Lhtp}{L_{\A^1}}
\nc{\Lzar}{L_{\Zar}}
\nc{\Lnis}{L_{\Nis}}
\nc{\Lmot}{L_{\mrm{mot}}}
\nc{\mot}{\mrm{mot}}
\nc{\SH}{\mbf{SH}}
\nc{\HH}{\mbf{H}}
\nc{\RR}{\mathbb{R}}
\nc{\CC}{\mathbb{C}}
\nc{\Mod}{\mrm{Mod}}

\nc{\MonUnit}{\mbf{1}}
\nc{\tr}{\on{tr}}
\nc{\vop}{\mrm{vop}}
\nc{\fr}{{\on{fr}}}
\nc{\Ar}{\mrm{Ar}}
\nc{\Vfr}{\on{Vfr}}
\nc{\frdiff}{{\on{frdiff}}}
\nc{\frGys}{\on{frGys}}
\nc{\SHfr}{\SH^{\fr}}
\nc{\SHfrdiff}{\SH^{\frdiff}}
\nc{\SHfrGys}{\SH^{\frGys}}
\nc{\InftyCat}{\infty\textnormal{-}\mrm{Cat}}
\nc{\TriCat}{\mathrm{TriCat}}
\nc{\Cat}{\mathrm{1\textnormal{-}Cat}}
\nc{\Th}{\on{Th}}

\nc{\CMon}{\mrm{CMon}{}}
\nc{\CAlg}{\mrm{CAlg}{}}
\nc{\MGL}{\mrm{MGL}}
\nc{\PMGL}{\mathrm{PMGL}}
\nc{\KGL}{\mrm{KGL}}
\nc{\kgl}{\mrm{kgl}}
\nc{\MSL}{\mrm{MSL}}
\nc{\MSp}{\mrm{MSp}}
\nc{\Seg}{\mrm{Seg}{}}
\nc{\Tw}{\mrm{Tw}}
\nc{\sslash}{/\mkern-6mu/}
\nc{\PrL}{\mrm{Pr}^\mrm{L}}
\nc{\PrR}{\mrm{Pr}^\mrm{R}}
\nc{\pr}{\mrm{pr}}
\nc{\efr}{\mrm{efr}}
\nc{\nfr}{\mrm{nfr}}
\nc{\dfr}{\mrm{fr}}
\nc{\tfr}{\mrm{tfr}}
\nc{\Vect}{\mathcal{V}\mrm{ect}}
\nc{\sVect}{\mrm{sVect}}
\nc{\Alg}{\mrm{Alg}}
\nc{\fix}{\mrm{fix}}
\nc{\Hilb}{\mathrm{Hilb}}
\nc{\flci}{\mathrm{flci}}
\nc{\Isom}{\mathrm{Isom}}
\nc{\GL}{\mathrm{GL}}
\nc{\BGL}{\mathrm{BGL}}
\nc{\PGL}{\mathrm{PGL}}
\nc{\SL}{\mathrm{SL}}
\nc{\Sp}{\mathrm{Sp}}
\nc{\fin}{\mathrm{fin}}
\nc{\cl}{\mathrm{cl}}
\nc{\cn}{\mathrm{cn}}
\nc{\sm}{\mathrm{sm}}
\nc{\heart}{\heartsuit}
\renewcommand{\1}{\mbf{1}}
\renc{\o}{\mrm{or}}
\nc{\GW}{\mrm{GW}}
\nc{\W}{\mrm{W}}
\nc{\ev}{\mrm{ev}}
\nc{\FSYN}{\mathcal{FS}\mrm{yn}}
\nc{\FFLAT}{\mathcal{FF}\mrm{lat}}
\nc{\mrk}{\mrm{mrk}}%
\nc{\FFmrk}{\mathcal{FF}\mrm{lat}^{\mrk}}
\nc{\FFnu}{\mathcal{FF}\mrm{lat}^{\mrm{nu}}}
\nc{\FFbas}{\mathcal{FF}\mrm{lat}^{\mrm{bas}}}
\nc{\FQSM}{\mathcal{FQS}\mathrm{m}}
\nc{\Quot}{\mathrm{Quot}}    %% Quot scheme
\nc{\COH}{\mathcal{C}\mathrm{oh}}
\let\phi\varphi

\let\emptyset\varnothing

\nc{\robber}{\mathcal{R}}%
\nc{\mv}{\mrm{mv}}
\nc{\const}{\mrm{const}}
\nc{\robbermv}{\robber^{\mv}}
\nc{\robberconst}{\robber^{\const}}
\nc{\robbernot}{\robber_0}%
\nc{\sectionmv}{i^{\mv}}%
\nc{\sectionconst}{i^{\const}}%
\nc{\sectionnot}{i}%
\nc{\st}{\mathrm{st}}

\nc{\Cofib}{\on{Cofib}}
\nc{\Fib}{\on{Fib}}
\nc{\initial}{\varnothing}
\nc{\op}{\mathrm{op}}
\nc{\bglt}{\mathrm{Tr}}
\nc{\ret}{\text{r{\'e}t}}
\nc{\adeg}{\mathrm{deg}^{\A^1}}
\nc{\aged}{\mathrm{ged}^{\A^1}}

\nc{\inftyCat}{\term{$\infty$-category}}
\nc{\inftyCats}{\term{$\infty$-categories}}

\nc{\inftyOneCat}{\term{$(\infty,1)$-category}}
\nc{\inftyOneCats}{\term{$(\infty,1)$-categories}}

\nc{\inftyGrpd}{\term{$\infty$-groupoid}}
\nc{\inftyGrpds}{\term{$\infty$-groupoids}}

\nc{\inftyTop}{\term{$\infty$-topos}}
\nc{\inftyTops}{\term{$\infty$-toposes}}

\nc{\inftyTwoCat}{\term{$(\infty,2)$-category}}
\nc{\inftyTwoCats}{\term{$(\infty,2)$-categories}}

\newcommand{\Witt}{\mathbf{W}}
\newcommand{\slicecomp}{\mathrm{sc}}
\newcommand{\s}{\mathrm{s}}
\newcommand{\f}{\mathrm{f}}
\newcommand{\Ext}{{\operatorname{Ext}}}
\newcommand{\MU}{\mathrm{MU}}
\newcommand{\BP}{\mathrm{BP}}
\newcommand{\comp}{{{\kern -.5pt}\wedge}}

\title{The Motivic Adams Conjecture}

\author[A. Ananyevskiy]{Alexey Ananyevskiy}
\address{LMU M\"unchen \\ Mathematisches Institut \\ Theresienstr. 39 \\ 80333 M\"unchen \\ Germany}
\email{\href{mailto:alseang@gmail.com}{alseang@gmail.com}}

\author[E. Elmanto]{Elden Elmanto}
\address{University of Toronto\\
	Department of Mathematics \\ 
        Toronto, Ontario, M5S 2E4\\ Canada}
\email{\href{mailto:elden.elmanto@utoronto.ca}{elden.elmanto@utoronto.ca}}
\author[O. R\"ondigs]{Oliver R\"ondigs}
\address{Universit\"at Osnabr\"uck \\ Institut f\"ur Mathematik \\ 49069 Osnabr\"uck \\ Germany\\ \&  Universitetet i Oslo \\ Matematisk Institutt \\ 0316 Oslo \\ Norway}
\email{\href{mailto:oliver.roendigs@uni-osnabrueck.de}{oliver.roendigs@uni-osnabrueck.de}}

\author[M. Yakerson]{Maria Yakerson}
\address{CNRS \\ Sorbonne Universit\'e \\ 4 place Jussieu\\ 75005 Paris, France} 

\email{\href{mailto:yakerson@imj-prg.fr}{yakerson@imj-prg.fr}}

\date{November 19th, 2024}

\begin{document}

\begin{abstract} We solve a motivic version of the Adams conjecture with the exponential characteristic of the base field inverted. In the way of the proof we obtain a motivic version of mod $k$ Dold theorem and give a motivic version of Brown's trick studying the homogeneous variety $(N_{\GL_r} T)\backslash\GL_r$ which turns out to be not stably $\A^1$-connected. We also show that the higher motivic stable stems are of bounded torsion.
\end{abstract}

\maketitle

\section{Introduction}
\label{sec:introduction}

        Adams' conjecture is a statement on the stable fibrewise equivalence of certain spherical fibrations associated to virtual real vector bundles over finite cell complexes.
        He formulated it as \cite[Conjecture 1.2]{adams.j1} with a clear goal in mind.
        George Whitehead's so-called $J$-homomorphism associates to a homotopy class $[f]$ of a continuous based function $f\colon S^r \to \mathrm{SO}(d)$ an element $J(f)\in \pi_{r+d}S^d$ involving the Hopf construction.
        It does not depend on $d$ if $d>r+1$, and can then be viewed as a homomorphism from the stable homotopy groups $\pi_r \mathrm{SO}(\infty)$ of the topological real $K$-theory spectrum to the stable homotopy groups $\pi_r\mathbb{S}$ of spheres.
        Adams succeeded to determine the image of this homomorphism in the series \cite{adams.j1,adams.j2,adams.j3,adams.j4} for all degrees $r\not\equiv 7(8)$, but only up to a factor of two if $r \equiv 7 (8)$; see \cite[Theorem 1.6]{adams.j4}.
        The validity of Adams conjecture removes this factor, and in particular shows that the image of the stable $J$-homomorphism is always a direct summand.
        
        Adams conjecture was proven in different ways in \cite{quillen.adams,sullivan.adams,quillen.adams2,friedlander,becker-gottlieb.adams,Brown73}.
        While Adams' series of articles on the $J$-homomorphism can be viewed as a systematic account on the production of a certain (nowadays called $v_1$-periodic) periodic family in the stable homotopy groups of spheres (itself a very prominent topic in algebraic topology), the solutions of Adams conjecture gave rise to amazing developments in mathematics --- the idea of localizations and completions of spaces, \'etale homotopy theory and the Becker-Gottlieb transfer, to name just a few.
        Perhaps none as striking as Quillen's computation of the algebraic $K$-theory of finite fields, itself a cornerstone of the creation of higher algebraic $K$-theory; see \cite[p.~140]{adams.infinite}.
        Higher algebraic $K$-theory constitutes a major motivating example for the $\A^1$-homotopy theory of Morel-Voevodsky started in \cite{MV}, which brings us to our current subject.

        Any vector bundle $\sE\to S$ over a Noetherian scheme gives rise to a ``spherical'' bundle $\sE-z(S) \to S$ by removing its zero section.
        Up to a simplicial suspension, this spherical bundle coincides with the Thom space $\Th(\sE):=\sE/\sE-z(S)$ of $\sE\to S$. 
        Viewing the Thom space as a $\wedge$-invertible object in the motivic stable homotopy category $\SH(S)$, and hence as a valid stable form of the spherical bundle $\sE-z(S)\to S$, leads to the following direct motivic analog of the complex Adams conjecture.
        Its formulation involves Adams operations $\psi^k$ on (virtual) vector bundles, as in the topological situation.

	\begin{thm*}[Theorems~\ref{thm:Adams} and~\ref{thm:Adams_at_char}]
		Let $\sE$ be a vector bundle over a regular scheme $S$ over a field $F$ of exponential characteristic $e$, and let $k\in \Z$ be an integer. Then for some $N\in \N_0$ one has $ \Th(k^N\otimes \sE) \cong \Th(k^N\otimes \psi^k \sE)$ in $\SH(S)[\tfrac{1}{e}]$. If $k$ is a power of $e$ then such an isomorphism exists in $\SH(S)$.
	\end{thm*}
	\noindent The same conclusion holds also for a singular scheme, provided that it admits a Jouanolou device; see Corollary~\ref{cor:singular}.
	
The Thom spectrum construction descends to a homomorphism $J\!:=\pi_0\! \Th \colon \!K_0(S) \rightarrow  \pi_0 \mathbf{Pic} \bigl(\SH(S)[\tfrac{1}{e}]\bigr)$ of abelian groups, thus the isomorphism in the Theorem above can be rephrased as an equality 
\begin{equation}\label{eq:adams}
k^NJ( \sV - \psi^k(\sV)) = 0 
\end{equation}
in the target abelian group, for any virtual vector bundle $\sV$ over $S$.
This equality appears in the familiar form of the Adams conjecture from algebraic topology.

\begin{rem}[Complex Betti realization and Adams conjecture] The complex Adams conjecture from algebraic topology can be deduced from the motivic Adams conjecture over the complex numbers in the following way. To obtain the complex Adams conjecture in algebraic topology, it suffices to prove the complex Adams conjecture over universal examples, that is, Grassmannians classifying complex vector bundles. More precisely, it suffices to prove the complex Adams conjecture over the topological space $\mathrm{Gr}_{\mathbb{C}}^{\mathrm{top}}(r,d)$ classifying rank $r$ quotients of trivial rank $d$ topological $\mathbb{C}$-vector bundles. This topological space is the analytification of the smooth projective $\mathbb{C}$-variety $\mathrm{Gr}_{\mathbb{C}}(r,d)$. Thus the validity of Theorem~\ref{thm:Adams} for $S =\mathrm{Gr}_{\mathbb{C}}(r,d)$ and $F = \mathbb{C}$ implies the usual complex Adams conjecture after applying the complex Betti realization functor $\mathrm{Betti}\colon \SH(\mathbb{C}) \rightarrow \SH$. The reason is that the diagram
\[
\begin{tikzcd}
K_0(\mathrm{Gr}_{\mathbb{C}}(r,d)) \ar{r} \ar[swap]{d}{\mathrm{Betti}} & \pi_0\mathbf{Pic} \bigl(\SH(\mathrm{Gr}_{\mathbb{C}}(r,d)\bigr) \ar{d}{\mathrm{Betti}}\\
KU^0( \mathrm{Gr}_{\mathbb{C}}^{\mathrm{top}}(r,d)) \ar{r} & \mathrm{Sph}\bigl(\SH(\mathrm{Gr}_{\mathbb{C}}^{\mathrm{top}}(r,d)\bigr)
\end{tikzcd}
\]
whose vertical arrows are induced by the complex Betti realization functor commutes. Here the terminal corner $\mathrm{Sph}(\SH(\mathrm{Gr}_{\mathbb{C}}^{\mathrm{top}}(r,d))$ is the monoid of spherical fibrations over $\mathrm{Gr}_{\mathbb{C}}^{\mathrm{top}}(r,d)$. The identity~\eqref{eq:adams} holds in the target of the top horizontal map. Since the Betti realization homomorphism on the left hand side is surjective (even an isomorphism), the identity holds in the target of the bottom horizontal map as well.
\end{rem}

\begin{rem}[The real motivic Adams conjecture] As stated in the introduction, the (more difficult) real Adams conjecture in topology addresses all real vector bundles, which are classified by classifying spaces of orthogonal groups and real topological K-theory. A motivic analog over a base scheme $S$ would involve the Hermitian $K$-theory of $S$ and a real analog of the $J$-homomorphism. We expect to address this in a sequel. 
\end{rem}

        At this point, the (perhaps too lengthy) historical introduction of the topological Adams conjecture must raise the question of applicability of this theorem in the context of computations of stable homotopy groups of motivic spheres.
        This, as well as other applications such as James periodicity for truncated projective spaces, will not be addressed in this paper, but instead in future work. 
        For now, we just offer a rather straightforward exercise to indicate how Theorem~\ref{thm:Adams} applies to torsion order questions of elements in stable homotopy groups of motivic spheres.
        The reader may check that the element $\nu\in \pi_{3,2}\1_F$ given by the second Hopf map $S^{7,4}\to S^{4,2}$ from the Hopf construction on $\SL_2=S^{3,2}$ is in the image of the motivic $J$-homomorphism
        \[ \Th\colon K_0(S^{4,2}) \to \pi_0\mathbf{Pic}(\SH(S^{4,2})) \]
        using a smooth affine quadric of dimension four as a model for the sphere $S^{4,2}$ \cite[Definition 4.7]{dugger-isaksen.hopf}. 
        In this case, the source $K_0(S^{4,2})$ contains a free abelian group of virtual rank zero bundles whose generator hits $\nu\in \pi_{3,2}\1_F$, with the latter group contained as a direct summand in the target.
        Theorem~\ref{thm:Adams} implies that the image $\nu$ must be torsion.
        With further work, the order can be determined as $24$ over fields of characteristic not two and not three, which fits with the computation in \cite[Theorem 5.5]{rso.oneline}.
        The situation for the next Hopf map $\sigma\colon S^{15,8}\to S^{8,4}$ and its surprising order in $\pi_{7,4}\1_{\mathbb{R}}$ observed in \cite{dugger-isaksen.real} will be discussed in future work.
        
        The reader with some experience in computations of motivic stable homotopy groups of spheres must have encountered several papers in which a candidate motivic spectrum for the image of the motivic $J$-homomorphism is proposed and its homotopy groups or sheaves are computed, such as \cite{bachmann-hopkins, bik, kong-quigley, boq}.
        While on its second page \cite{bik} explicitly asks about the motivic Adams conjecture because of a certain insecurity regarding the correctness of their candidate over the real numbers, they ``make no attempt to study these more geometric issues''.
        These issues are studied here, at least up to inverting the characteristic of the base field if it is positive.
	Regarding this slight defect, note that if $k$ is a power of the positive characteristic of $F$, then one may use the Frobenius homomorphism  to avoid the inversion of the characteristic; see Theorem~\ref{thm:Adams_at_char} for the details. Furthermore, if one shows that the motivic suspension spectrum $\Sigma_\T^\infty ((N_{\GL_r} T)\backslash \GL_r)_+$ of the homogeneous variety with respect to the normalizer of the standard maximal torus  $T\le \GL_r$ is strongly dualizable in $\SH(F)$, then the conclusion of Theorem~\ref{thm:Adams} also holds without inversion of the characteristic of $F$ with the same proof.
	
	The proof of Theorem~\ref{thm:Adams} consists of the following steps.

	\textbf{(1)} Reduction to the case where the vector bundle $\sE\to X$ has the normalizer $N_{\GL_r} T$ of the standard maximal torus as the structure group.
        For this we use motivic Becker--Gottlieb transfers introduced by Levine~\cite{levine.euler} -- with \cite{HoyoisGLV} as a precursor -- and a motivic version of Brown's trick \cite{Brown73} (see e.g. \cite[Proposition~5.1]{Ebert} for an exposition of Brown's approach to the Adams conjecture).
        In order to transplant Brown's trick to the motivic realm, we study the Becker--Gottlieb--Hoyois transfer for the variety $(N_{\GL_r} T)\backslash \GL_r$; see Lemma~\ref{lem:Brown_trick_adhoc}, which surprisingly turns out to be quite intricate. In particular, this variety is not stably $\A^1$-connected; see Remark~\ref{rem:nonconnected}.
        Perhaps another surprise to the reader may be that we transfer the ``most topological'' of the proofs of the classical Adams conjecture to the motivic situation, although there are more algebro-geometric approaches among those listed above.

	\textbf{(2)} Reduction to the case of a line bundle. This is accomplished via \'etale transfers in the motivic stable homotopy theory; see Lemma~\ref{lem:Adams_away_reduced}. It uses the observation that a vector bundle with the structure group  $N_{\GL_r} T$ is the direct image of a line bundle along a suitable Galois cover. 
	
	\textbf{(3)} The case of a line bundle. Here we explicitly construct morphisms $\Th(\sL)\to \Th(\sL^{\otimes k})$ for $k$ odd and $\Th(k\otimes \sL)\to \Th(k\otimes \sL^{\otimes k})$ for $k$  even of particular $\A^1$-degrees, based on the morphisms of motivic spheres considered in Lemma~\ref{lem:local_degree}.
        To these we apply the following motivic version of the mod $k$ Dold theorem.
	\begin{thm*}[Theorem~\ref{thm:modkDold}]
		Let $S$ be a connected regular scheme over a field $F$, let $\sE_1, \sE_2$ be vector bundles over $S$ of the same rank and $k\in \mathbb{Z}$ . Suppose that there exists a morphism $\phi\in \Hom_{\SH(S)}( \Th(\sE_1), \Th(\sE_2))$ of $\A^1$-degree $k\in \Z\subseteq \GW(F(S))$. Then for some $N\in \N_0$ one has $\Th(k^N\otimes \sE_1) \cong \Th(k^N\otimes \sE_2)$ in $\SH(S)$.
	\end{thm*}
	\noindent
	The proof of the mod $k$ Dold theorem is in turn based on the observation that two Thom spaces over connected regular $F$-schemes are stably isomorphic if and only if there exists a morphism between them such that its $\A^1$-degree at the generic point is invertible.
        Such an isomorphism clearly exists over an open subset of the base (e.g. over which both Thom spaces are trivial) and we iteratively extend this isomorphism to the whole base taking additive multiples of the considered vector bundles to eliminate the arising obstructions.
        This requires a delicate analysis of the arising localization sequences and the following ingredient.
	\begin{lem*}[Lemma~\ref{lm:kernel_nilpotent}]
		Let $S$ be a regular scheme over a field $F$, $A\in\SH(F)$ be a commutative ring spectrum, $\Vv$ be a virtual vector bundle over $S$, $U\subseteq S$ be a dense open subset and $\alpha\in A^{p}(S;\Vv)$ be such that $\alpha|_U=0$ (see Definition~\ref{def:twist} for the notation). Then there exists $n\in \N$ such that $\alpha^n=0$.
	\end{lem*}

	\noindent The proof of this lemma is based on Morel's connectivity theorem~\cite[Theorem~6.4.1]{morel-connectivity}.

	In case of $k=p^{\pm n}$ being a power of the characteristic of the base field one may use the Frobenius homomorphism in place of the Brown's trick, effectively skipping the step \textbf{(1)} above. In particular, this allows to avoid inversion of the characteristic, see Theorem~\ref{thm:Adams_at_char} for the details.
	
	One of the ingredients in Brown's original approach to the proof of Adams conjecture is finiteness of the higher homotopy groups of the sphere spectrum obtained by Serre in his thesis. Although our argument proceeds in a slightly different manner and avoids this step, we prove the following motivic version of this finiteness property, which may be of independent interest.
	
	\begin{thm*}[Theorem~\ref{thm:bounded_exponent}]
		Let $F$ be a field of exponential characteristic $e$. Then for $0<s,w\in \bN$ there exists a natural number $N$ (depending only on $s$ and on $w$, and not on $F$) such that the abelian group $\pi_{s+(w)}\1_F[\tfrac{1}{e}]=\Hom_{\SH(F)}(\Sigma_{\mathrm t}^w \Sigma_{\mathrm{s}}^s \1_F, \1_F)\otimes_\Z \Z[\tfrac{1}{e}]$ is $N$-torsion.
	\end{thm*}	
	\noindent This theorem is obtained combining the finiteness of stable stems in topology, delicate analysis of the slice spectral sequence for the motivic sphere spectrum and the recent results of Bachmann and Hopkins on the $\eta$-inverted motivic sphere spectrum \cite{bachmann-hopkins}.
	
	\textbf{Acknowledgements.} Alexey Ananyevskiy is supported by the DFG Heisenberg grant AN 1545/1-1 and the DFG research grant AN 1545/4-1. Oliver R\"ondigs is supported by the DFG project ``Algebraic bordism spectra'' and thanks Ben Williams for mentioning James periodicity, and Karl-Heinz Knapp for a copy of \cite{Brown73}. Elden Elmanto acknowledges support by an Erik Ellentuck fellowship while at the Insitute for Advanced Study and thanks Adeel Khan for conversations related to this project many years ago. Maria Yakerson is grateful to CNRS and Sorbonne Universit\'e for generous and supportive work conditions. We thank Toni Annala for interesting questions related to this project and Oscar Randal-Williams for the questions which eventually led to the improvement of the main results of the paper, almost allowing to remove the inversion of the exponential characteristic. We also thank the anonymous referees for the detailed reading of our manuscript and helpful remarks.

	Throughout the paper we employ the following assumptions and notations.
	
	\begin{tabular}{l|l}
	  $S$ & a scheme \\	
	  $\Sch_S, \Sm_S$ & the category of schemes, resp. smooth schemes, over $S$\\		  
	  & all schemes are assumed to be Noetherian\\		
	  $F$ & a field \\
	  $\Vv$ & a virtual vector bundle, that is a formal difference $\sE_1\ominus \sE_2$ of vector bundles \\	
  	  $n\otimes \Vv$ & $\sE_1^{\oplus n}\ominus \sE_2^{\oplus n}$ for $n\in \N_{\ge 0}$, resp. $\sE_2^{\oplus -n}\ominus \sE_1^{\oplus -n}$ for $n\in \Z_{<0}$, and $\Vv=\sE_1\ominus \sE_2$\\
	  $\SH(S)$ & the motivic stable homotopy category over $S$ \cite{MV,Voevodsky:1998}\\
	  $\1_S$ & the motivic sphere spectrum over the scheme $S$\\	
	  $\T$ & $\A^1/(\A^1-\{0\})$ considered as a pointed motivic space over some base scheme $S$\\
	  $\Th(\sE)$ & the Thom spectrum $\Sigma^\infty_\T (\sE/(\sE-S))\in \SH(S)$ of a vector bundle $\sE$ over $S$ \\
	  $\Sigma^{p}_{\mathrm{s}}$, $\Sigma^{p}_{\mathrm{t}}$, $\Sigma^{p}_\T$ & $p$-fold simplicial, $(\Gm,1)$ and $T$-suspensions respectively \\
	  $\GW(F)$ & the Grothendieck--Witt ring of nondegenerate symmetric bilinear forms over $F$
	\end{tabular}

	\section{Preliminaries on cohomology theories}\label{sec:prel-cohom-theor}
        This section contains several definitions, conventions, and results regarding cohomology theories represented by motivic spectra, that is, objects in the Morel-Voevodsky $\P^1$-stable $\A^1$-homotopy category $\SH(S)$.
        The main result is Corollary~\ref{cor:kernel_nilpotent}, a nilpotence statement for elements in a suitable cohomology theory over a connected regular scheme over a field restricting to zero on its generic point. 
        Base change is relevant for that statement, thus we are going to freely employ the six functor formalism on the motivic stable homotopy categories as in \cite{Ayoub, Ayoub2}.
	As more concise references, the reader may consider \cite[Section 2]{HoyoisGLV} or \cite{hoyois-sixops}. We briefly recall some properties and fix notation.
	
	\begin{defn}\label{def:six-functors}
	Let $f\colon Y\to X$ be a morphism of schemes. There is an adjoint pair of functors
	\[
	f^*\colon \SH(X) \leftrightarrows \SH(Y) \colon f_*.
	\]
	If $f$ is smooth, then one also has an adjoint pair of functors
	\[
	f_\sharp\colon \SH(Y) \leftrightarrows \SH(X) \colon f^*.
	\]
	We usually denote the corresponding units and counits of the adjunctions as
	\[
	\id_{\SH(X)} \xrightarrow{\rho} f_\ast f^\ast,\quad f^\ast f_\ast \xrightarrow{\omega} \id_{\SH(Y)},\quad 
	\id_{\SH(Y)} \xrightarrow{\lambda} f^\ast f_\sharp, \quad f_\sharp f^\ast\xrightarrow{\mu} \id_{\SH(X)}.
	\]
	The functor $f^\ast$ is strong symmetric monoidal, whence $f_\ast$ is lax symmetric monoidal and $f_\sharp$ is op-lax symmetric monoidal. In particular, $f^\ast \1_X\simeq \1_Y$ and for $A,B\in \SH(Y)$ there is a natural morphism
	\[
	\Delta \colon f_\sharp (A\wedge B) \to f_\sharp A\wedge f_\sharp B
	\]
	adjoint to the morphism
	\[
	\lambda_A \wedge \lambda_B\colon A\wedge B \to f^*(f_\sharp A\wedge f_\sharp B) \cong f^*f_\sharp A\wedge f^*f_\sharp B.
	\]
	For $A\in \SH(Y)$ and $B\in \SH(X)$ the projection formula \cite[Theorem~6.18(7)]{hoyois-sixops} yields an isomorphism $f_\sharp (A\wedge f^*B) \xrightarrow{\simeq}  f_\sharp A\wedge B$ that is adjoint to
	\[
	\lambda_A \wedge \id_{f^*B} \colon A\wedge f^*B \to f^*(f_\sharp A\wedge B) \cong f^*f_\sharp A\wedge f^*B
	\]
	If $B$ is $\wedge$-invertible, then there is also an isomorphism $f_\ast A\wedge B  \xrightarrow{\simeq} f_\ast(A\wedge f^*B) $ adjoint to
	\[
	\omega_A \wedge \id_{f^*B}\colon f^\ast f_* A\wedge f^*B \cong f^\ast ( f_* A\wedge B ) \xrightarrow{\simeq} A\wedge f^*B.
	\]
        \end{defn}	

        Among the many consequences of Ayoub's work on the six functor formalism is the extension of the Thom spectrum construction from vector bundles to classes in $K_0$, that is, virtual vector bundles. 
        
	\begin{defn}\label{def:thom-spectrum}
		Let $\Vv=\sE_1\ominus \sE_2$ be a virtual vector bundle over $S$ that is a formal difference of vector bundles over $S$. The \textit{Thom spectrum of $\Vv$} is
		\[
		\Th(\Vv)=\Th(\sE_1)\wedge  \Th(\sE_2)^\vee \in \SH(S)
		\]
		where $	\Th(\sE_2)^\vee = \underline{\Hom}_{\SH(S)}(\Th(\sE_2),\1_S)$ is the object dual to $\Th(\sE_2)$ in $\SH(S)$.
	\end{defn}

        Thom spectra allow a reasonable formulation of twisted generalized cohomology theories, as in \cite[Definition~2.2.1.(ii)]{DJKFundamental}.
        
	\begin{defn} \label{def:twist}
	  Let $X\in \Sch_S$ with the structure morphism $f\colon X\to S$ and let $\Vv = \sE_1\ominus \sE_2$ be a virtual vector bundle over $X$.
          For $A\in\SH(S)$ and $p\in \Z$ we set
		\[
		A^{p}(X;\Vv) := \Hom_{\SH(X)} (\1_X, \Sigma_{\mathrm s}^{p} f^*A 	\wedge \Th(\Vv)).
		\]
		For $Y\in \Sch_X$ with the structure morphism $g\colon Y\to X$ we let
		$A^{p}(Y;\Vv) := A^{p}(Y;g^*\Vv)$. For a closed embedding $i\colon Z\to X$ we denote
		\[
		A^{p}_Z(X;\Vv) := \Hom_{\SH(X)} (i_*\1_Z, \Sigma_{\mathrm s}^{p} f^*A \wedge \Th(\Vv))
		\]
		the cohomology supported on $Z$. For $Z=X$ one has $A^{p}_X(X;\Vv)=A^{p}(X;\Vv)$. The shortened form $A^{p}_Z(X)=A^{p}_Z(X;0)$ will be used for the zero (virtual) vector bundle.
        \end{defn}

        The localization theorem in motivic homotopy theory (see e.g. \cite[Theorem~6.18]{hoyois-sixops}) yields an exact sequence
		\[
		A^{p}_{Z}(X;\Vv) \xrightarrow{} A^{p}(X;\Vv) \xrightarrow{} A^{p}(X-Z;\Vv) \xrightarrow{} A^{p+1}_{Z}(X;\Vv).
		\]
                A trivialization $\theta \colon \Th(\Vv) \xrightarrow{\simeq} \Sigma_\T^{r} \1_X$ induces an isomorphism
		\[
		\Theta := ((\id \wedge \theta) \circ -) \colon A^{p}_Z(X;\Vv) \xrightarrow{\simeq} A^{p}_Z(X;r\otimes \sO_X).
		\]
		We will always denote such isomorphisms by the corresponding capital letters, i.e.~a trivialization $\theta$ induces an isomorphism $\Theta$, a trivialization $\upsilon$ induces an isomorphism $\Upsilon$, and so on.
		Note that if $X$ is the spectrum of a field or of a local ring, then one always has a trivialization $\theta \colon \Th(\Vv) \xrightarrow{\simeq} \Sigma_\T^{r} \1_X$ with $r=\rk \Vv = \rk \sE_1 - \rk \sE_2$.
		
		For a commutative ring spectrum $A\in\SH(S)$ and virtual vector bundles $\Vv, \Vv'$ over $S$, smash product together with the canonical isomorphism $\Th(\Vv\oplus \Vv')\cong \Th(\Vv)\wedge \Th(\Vv')$ gives rise to an associative bilinear pairing
		\[
		A^{p}(X;\Vv)\times A^{q}(X;\Vv') \to A^{p+q}(X;\Vv\oplus \Vv').
		\]
		This endows $\bigoplus_{n\ge 0} A^{*}(X;n\otimes \Vv)$ with a ring structure. A trivialization $\theta\colon \Th(\Vv) \xrightarrow{\simeq} \Sigma_T^r\1_X$ induces a family of isomorphisms $\Theta\colon A^{*}(X;n\otimes \Vv) \xrightarrow{\simeq} A^{*}(X;nr\otimes \sO_{X})$ and it is straightforward to see that these combine into an isomorphism of rings
		\[
		\Theta\colon \bigoplus_{n\ge 0} A^{*}(X,n\otimes \Vv) \xrightarrow{\simeq} \bigoplus_{n\ge 0}  A^{*}(X; nr \otimes \sO_{X}).
		\]
		If $r = 0$ then abusing the notation we also denote $\Theta$ the composition
		\[
		\Theta\colon \bigoplus_{n\ge 0} A^{*}(X,n\otimes \Vv) \xrightarrow{\simeq} \bigoplus_{n\ge 0}  A^{*}(X) \xrightarrow{sum} A^{*}(X)
		\]
		that is a homomorphism of rings.
				
		Recall that by \cite[Theorem~6.4.1]{morel-pi0} and \cite[Corollary~1.25]{Morel} for a field $F$ there is a natural isomorphism of rings
		\[
		\1^{0}(\Spec F)\cong \GW(F),
		\]
		with $\GW(F)$ being the Grothendieck--Witt ring of nondegenerate symmetric bilinear forms over $F$. Let $\Vv$ be a virtual vector bundle of rank $0$ over $S$ and $f\colon \Spec F \to S$ be a regular morphism, where $F$ is a field. By the above a trivialization $\theta\colon \Th(f^*\Vv)\xrightarrow{\simeq} \1_F$ gives rise to a homomorphism of rings
		\[
		\Theta\colon \bigoplus_{n\ge 0} \1^{0}(\Spec F;n\otimes \Vv) \to \GW(F).
		\]
		We will usually abuse the notation writing $\Theta$ also for the isomorphisms that are components of this ring homomorphism.

	\begin{rem} \label{rem:J}
		Similar to the setting of the equivariant stable homotopy theory, a natural grading on a cohomology theory representable in $\SH(S)$ is given by the Picard group $\Pic(\SH(S))$ of $\wedge$-invertible objects in $\SH(S)$, but this group is rather mysterious. One has a much better understood substitute $\Z\times K_0(S)$ that comes with a homomorphism
		\[
		\Z\times K_0(S) \to \Pic(\SH(S)),\quad (n,\Vv)\mapsto \Sigma_{\mathrm s}^n \Th(\Vv)
		\]
		and we use it here as the grading. Note that the above homomorphism is in general neither injective (in particular, by the Adams conjecture that we discuss in the current paper, but visible already in \cite[Prop.~2.2]{rondigs.theta}) nor surjective (see for example \cite{hu.picard} and \cite{bachmann.inv-motives}). Following \cite[Section 16.2]{norms}, one may enhance this homomorphism to the \textit{motivic} $J$-\textit{homomorphism}
		\[
		\mathbf{Th}\colon \mathbf{K}(S) \to \mathbf{Pic}(\SH(S))
		\]
		from the Thomason-Trobaugh $K$-theory space $\mathbf{K}(S)$ of $S$ to the Picard space of $\SH(S)$. On path components it induces the above homomorphism
		$\Th\colon K_0(S) \to \Pic(\SH(S))$, but it also induces homomorphisms of higher homotopy groups.
		The path component of $\mathbf{Pic}(\SH(S))$ containing the basepoint $\1_S\in \mathbf{Pic}(\SH(S))$ coincides with the classifying space of the group-like simplicial monoid of all self-equivalences of $\1_S$.
                The space of all endomorphisms of $\1_S$ on the one hand coincides with the global sections of the infinite $\P^1$-loop space of $\1_S$, and on the other hand contains a path component containing $\id_{\1_S}$, the identity on $\1_S$.
                This path component coincides with the path component of $\id_{\1_S}$ in the simplicial monoid of all self-equivalences of $\1_S$.
                In particular, $\pi_1\mathbf{Pic}(\SH(S))$ is the group of units in $\1^0(S)$, and $\pi_n\mathbf{Pic}(\SH(S))\cong \1^{1-n}(S)$ for $n>1$.
                Hence the motivic $J$-homomorphism induces homomorphisms
		\[
		K_1(S)\to (\1^0(S))^\times,\quad K_n(S)\to \1^{1-n}(S),\, n\ge 2
		\]
                from the algebraic $K$-groups to the motivic stable homotopy groups
                of spheres over $S$.
	\end{rem}

	\begin{rem} \label{rem:local_purity}
		Suppose that $S$ is a regular scheme over a field $F$. Then \cite[Proposition~4.3.10~(ii)]{DJKFundamental} (see also \cite[Theorem~C.1]{DFJK21}) yields that $S$ satisfies local purity (cf. \cite[Conjecture~B]{DegliseOrientation}), i.e. for a point $\zeta\in S$ and a virtual vector bundle $\Vv$ over $S$ one has purity isomorphisms 
		\[
		\1^p_{\{\zeta\}}(V;\Vv)\cong \1^p_{\{\zeta\}}(V;r\otimes \sO_V) \cong \1^p(\zeta;(r-c)\otimes \sO_\zeta),
		\]
		where $V=\Spec \sO_{S,\zeta}$ $r=\rk \Vv$, $c=\dim V$. Here the first isomorphism is induced by a trivialization of $\Th(\Vv|_{V})$, and the second isomorphism is induced by isomorphism \cite[Definition~4.3.7]{DJKFundamental} and a trivialization of $\Th(N_{\zeta/V})$. 
	\end{rem}
	
	\begin{lem} \label{lm:kernel_nilpotent}
		Let $S$ be a regular scheme over a field $F$,  $A\in\SH(F)$ be a commutative ring spectrum, $\Vv$ be a virtual vector bundle over $S$, $U\subseteq S$ be a dense open subset and $\alpha\in A^{p}(S;\Vv)$ be such that $\alpha|_U=0$. Then there exists $n\in \N$ such that $\alpha^n=0$.
	\end{lem}
	\begin{proof}
		The justification of the lemma proceeds by induction on the size of a minimal covering $S=\bigcup_{i=1}^m W_i$ with $W_i\subseteq S$ being affine open. 
		
		Suppose first that $S$ is affine. Changing $F$ to its prime subfield we may assume that $F$ is perfect. Popescu's theorem \cite[Theorem~1.1]{Spivakovsky99} yields that $S$ is a filtered limit of smooth schemes over $F$, thus by continuity property of the stable motivic homotopy category \cite[Proposition~C.12(4)]{HoyoisGLV} we may assume that $S$ is smooth over $F$, with $f\colon S\to \Spec F$ being its structure morphism. Given $U$ as in the statement of the lemma, set $Z:=S-U$. Let $i\colon Z\to S$ be the resulting closed embedding and consider the isomorphisms
		\begin{gather*}
		A^{np}(S;n\otimes \Vv)=\Hom_{\SH(S)}(\1_S,\Sigma_{\mathrm s}^{np} f^*A\wedge \Th(n\otimes \Vv)) \cong \Hom_{\SH(S)}(\Th(-n\otimes \Vv),\Sigma_{\mathrm s}^{np} f^*A),\\
		\begin{multlined}
		A^{p}_Z(S;\Vv) = \Hom_{\SH(S)}(i_*\1_Z,\Sigma_{\mathrm s}^p f^*A \wedge \Th(\Vv))\cong \Hom_{\SH(S)}(i_*\1_Z\wedge  \Th(-\Vv),\Sigma_{\mathrm s}^p f^*A) \cong \\ \cong \Hom_{\SH(S)}(i_*(\1_Z\wedge  i^*\Th(-\Vv)),\Sigma_{\mathrm s}^p f^*A) \cong  \Hom_{\SH(S)}(i_*i^* \Th(-\Vv),\Sigma_{\mathrm s}^p f^*A).
		\end{multlined}
		\end{gather*}
		Here the isomorphisms are given by $\wedge$-invertibility of $\Th(n\otimes \Vv)$ and $\Th(\Vv)$, see also the end of Definition~\ref{def:six-functors} for the isomorphism $i_*\1_Z\wedge  \Th(-\Vv)\cong i_*(\1_Z\wedge  i^*\Th(-\Vv))$.
        By the localization exact sequence there exists $\overline{\alpha}\in A^{p}_Z(S;\Vv)$ which maps to $\alpha$ under the extension of support. Slightly abusing notation, we denote in the same way the elements that one obtains under the isomorphisms given above.
                Consider the following commutative diagram: 
		\[
		\begin{tikzcd}
		\Th(-n\otimes \Vv) \ar[d,"\cong"] \ar[r,"{\lambda}"] & f^*f_\sharp \Th(-n\otimes \Vv) \ar[d,"\cong"]  \\
		\underbrace{\Th(-\Vv) \wedge \dots \wedge \Th (-\Vv) }_n \ar[r,"\lambda"] \ar[d,"\rho\wedge\dots\wedge\rho"] & f^* f_\sharp(\underbrace{\Th(-\Vv) \wedge \dots \wedge \Th (-\Vv) }_n) \ar[d,"f^*(\Delta)"] \\
		\underbrace{i_*i^*\Th(-\Vv) \wedge \dots \wedge i_*i^*\Th (-\Vv) }_n \ar[dr,"\lambda\wedge\dots\wedge \lambda"] \ar[d,"\cup \circ \overline{\alpha}\wedge\dots\wedge \overline{\alpha}"] & f^* (\underbrace{f_\sharp \Th(-\Vv) \wedge \dots \wedge f_\sharp\Th (-\Vv) }_n) \ar[d,"f^*f_\sharp(\rho)\wedge\dots\wedge f^*f_\sharp(\rho)"] \\
		\Sigma_{\mathrm s}^{np} f^*A & \underbrace{f^*f_\sharp i_*i^*\Th(-\Vv) \wedge \dots \wedge f^*f_\sharp i_*i^*\Th (-\Vv) }_n \ar[l,"\cup \circ f^*(\widetilde{\alpha}\wedge\dots\wedge \widetilde{\alpha})"]
		\end{tikzcd}
		\]
		Here $\lambda$ and $\rho$ are units of the respective adjunctions, the vertical isomorphisms in the first row are the canonical ones, $\Delta$ is induced by the op-lax monoidality of $f_\sharp$, and $\widetilde{\alpha}\colon f_\sharp i_*i^* \Th (-\Vv) \to \Sigma_{\mathrm s}^p A$ corresponds to $\overline{\alpha}$ under the adjunction $f_\sharp \dashv f^*$. Note that the composition along the  left column is $\alpha^n$. We claim that there exists a natural number $n$ such that
		\[
		\Hom_{\SH(F)}(f_\sharp \Th(-n\otimes \Vv), \underbrace{f_\sharp i_*i^*\Th(-\Vv) \wedge \dots \wedge f_\sharp i_*i^*\Th (-\Vv) }_n)=0
		\]
		This would imply that the composition along the right column is $0$, whence $\alpha^n=0$. Since $S$ is affine then $\Th(-\Vv)\cong \Sigma_\T^{-r}\Th(\sE)$ for a vector bundle $\sE$ over $S$ and an integer $r$. Suspension $\Sigma_\T^{-r}$ commutes with $f_\sharp$, $i^*$ and $i_*$, yielding
		\begin{multline*}
		\Hom_{\SH(F)}(f_\sharp \Th(-n\otimes \Vv), f_\sharp i_*i^*\Th(-\Vv) \wedge \dots \wedge f_\sharp i_*i^*\Th (-\Vv) )\cong\\
		\cong\Hom_{\SH(F)}(f_\sharp \Th(n\otimes \sE), f_\sharp i_*i^*\Th(\sE) \wedge \dots \wedge f_\sharp i_*i^*\Th (\sE)).
		\end{multline*}
		Let $j\colon Z\to \sE$ be the composition of the embedding $i\colon Z\to S$ and the zero section $S\to \sE$. Then $ f_\sharp i_*i^*\Th (\sE) \cong \Sigma_\T^\infty\sE/(\sE-j(Z))$, so it suffices to check that
		\[
		\Hom_{\SH(F)}(\Sigma_\T^\infty ( \sE^{\oplus n} / (\sE^{\oplus n} - S) ), \Sigma_\T^\infty (\underbrace{\sE\times \sE \times \dots \times\sE}_n / (\underbrace{\sE\times \sE \times \dots \times\sE}_n- \Delta_Z (Z))))=0,
		\]
		where $\Delta_Z \colon Z \to \sE\times \sE \times \dots \times\sE$ is the diagonal embedding induced by $j\colon Z\to \sE$. It follows from \cite[Theorem~6.4.1]{morel-connectivity} that 
		\[
		\Sigma^\infty_\T (\underbrace{\sE\times \sE \times \dots\times \sE}_n / (\underbrace{\sE\times \sE \times \dots\times \sE}_n- \Delta_Z (Z))) \in \SH(F)_{\ge ns+nc}
		\]
		for $s=\rk \sE$ and $c=\codim_S Z$; here the subscript $\ge ns+nc$ refers to Morel's homotopy $t$-structure \cite[Section~5.2]{Morel:2003}. If $ns+nc>\dim \sE^{\oplus n}=ns+d$ or, equivalently, $n>\frac{\dim S}{c}$, then the coniveau spectral sequence (see, for example, \cite[Definition~3.2]{ana-zero}) yields
		\[
		\Hom_{\SH(F)}\bigl(\Sigma_\T^\infty ( \sE^{\oplus n} / (\sE^{\oplus n} - S) ), \Sigma_\T^\infty (\underbrace{\sE\times \sE \times \dots\times \sE}_n / (\underbrace{\sE\times \sE \times \dots \times\sE}_n- \Delta_Z (Z)))\bigr)=0
		\]
		as claimed above\footnote{Here are more details: Since the domain of the homomorphism group in question is a cofiber of two smooth schemes, it suffices to prove the vanishing result for smooth schemes. The connectivity of the target in the homotopy $t$-structure and the boundedness of Nisnevich cohomological dimension of finite dimensional Noetherian schemes implies that the cohomology of a scheme with coefficients in the homotopy sheaves of the target vanishes for $n$ large enough. The global statement then follows from the coniveau spectral sequence (which also manifests as the Nisnevich descent spectral sequence).}. This proves the lemma if $S$ is affine, providing the induction start.
		
		For the induction step, suppose that $S=\bigcup_{i=1}^m W_i$ with $W_i\subseteq S$ being affine open. Assume $U$ and $\alpha$ are in the statement of the lemma. Set $W=\bigcup_{i=1}^{m-1} W_i$. By induction and the induction start, there exist $N\geq n\in \N$ such that $(\alpha|_W)^N=0$ and $(\alpha|_{W_m})^n=0$. Since then also $(\alpha|_{W_m})^N=0$, the Mayer--Vietoris exact sequence
		\[
		A^{p-1}(W\cap W_m;\Vv) \xrightarrow{\partial} A^{p}(S;\Vv) \to A^{p}(W;\Vv)\oplus A^{p}(W_m;\Vv)
		\]
		provides $\beta\in A^{p-1}(W\cap W_m;\Vv)$ with $\partial{\beta}=\alpha^N$. Then $\alpha^{N+n} = (\partial{\beta}) \cdot \alpha^n = \partial (\beta \cdot \alpha^n|_{W\cap W_m})=0$ since $\alpha^n|_{W\cap W_m} = (\alpha^n|_{W_m})|_{W\cap W_m}=0$, concluding the induction step and thus the proof.
	\end{proof}
	
	\begin{cor} \label{cor:kernel_nilpotent}
		Let $S$ be a connected regular scheme over a field $F$ with the generic point $\xi \in S$, $\Vv$ be a virtual vector bundle over $S$ and $A\in\SH(F)$ be a commutative ring spectrum. Then the kernel
		\[
		\ker (A^{p}(S;\Vv)\to A^{p}(\xi;\Vv))
		\]
		consists of nilpotent elements.
	\end{cor}
	\begin{proof}
		The statement follows from Lemma~\ref{lm:kernel_nilpotent} since
		\[
		\ker (A^{p}(S;\Vv) \to A^{p}(\xi;\Vv))=\bigcup_{\substack{\emptyset\neq U\subseteq S\\ \text{open}}} \ker (A^{p}(S;\Vv) \to A^{p}(U;\Vv) )
		\]
                using the continuity property of the stable motivic homotopy category \cite[Proposition~C.12(4)]{HoyoisGLV}.
	\end{proof}

	\section{A digression on the degree}\label{sec:digression-a1-degree}
        In order to provide equivalences of Thom spectra, a degree criterion detecting such equivalences will be provided in Lemma~\ref{lem:generic_iso}.
        Besides that, the motivic discrepancy between the degree of an Adams operation (which is an integer) and the degree of a map of motivic spheres (which is a class in the Grothendieck-Witt) ring will be moderated with the help of Lemma~\ref{lem:local_degree}.

	\begin{defn}\label{def:a1-degree}
	  Let $F$ be a field. A morphism 
	  $\phi\in \Hom_{\SH(F)}(\Sigma^\infty_\T \A^n/(\A^n- \{0\}), \Sigma^\infty_\T \A^n/(\A^n- \{0\}))$ is said to be of $\A^1$-\textit{degree} $q\in \GW(F)$ if $\Sigma^{-n}_\T \phi$ corresponds to $q$ under the isomorphism 
	  \[
	  \phi \in \Hom_{\SH(F)}(\Sigma^\infty_\T \A^n/(\A^n- \{0\}),\Sigma^\infty_\T \A^n/(\A^n- \{0\})) \cong \Hom_{\SH(F)}(\1_F,\1_F) \cong \GW(F)\ni q
	  \]
          given by \cite[Corollary~1.25]{Morel}.
          A morphism $\phi\colon \A^n/(\A^n- \{0\}) \to \A^n/(\A^n- \{0\})$ in the unstable motivic category over $F$ is of $\A^1$-\textit{degree} $q\in \GW(F)$ if $\Sigma^\infty_\T \phi$ is of $\A^1$-degree $q$.
          In both cases the abbreviation $q=\adeg(\phi)$ may be used.        
	\end{defn}

        Notation for bilinear forms in $\GW(F)$ will follow standard rules, as in \cite{ekm} or \cite{Scharlau85}.
        In particular, the rank 1 symmetric bilinear form associated with a unit $u\in F$ is denoted $\langle u\rangle$. Roughly speaking, the next statement says that any integer $k \in \mathbb{Z} \subset GW(F)$ can be written as $p q$ where $q \in GW(F)$ and $p$ is realized as the $\mathbb{A}^1$-degree of a map between spheres of geometric origin. This should be thought of as a motivic analog of the fact that any integer can be realized as the degree of a self-map between spheres. 
        
	\begin{lem} \label{lem:local_degree}
		Let $F$ be a field and $k\in \N$.
		\begin{enumerate}
			\item 
			If $k$ is odd, then there exists $q\in\GW(F)$ such that $\adeg(\phi) \cdot q=k\in \Z\subseteq GW(F)$, where 
			$\phi\colon \A^1/(\A^1 - \{0\})\to \A^1/(\A^1 - \{0\})$ is defined via $\phi(x):=x^k$.
			\item
			If $k$ is even then there exist $M\in \N_0$, degree $k$ homogeneous polynomials 
			\[
			f_1,f_2,\hdots,f_k\in \Z[x_1,x_2,\hdots,x_k]
			\]
			with $Z(f_1,f_2,\hdots,f_k)=\{0\}\in \A^k$ as a set and $q\in GW(F)$ such that $\adeg(\phi)\cdot q=k^M\in \Z\subseteq GW(F)$, where 
			\[
			\phi\colon \A^k/(\A^k- \{0\}) \to \A^k/(\A^k- \{0\}),\quad (x_1,x_2,\hdots,x_k)\mapsto (f_1,f_2,\hdots,f_k).
			\]
		\end{enumerate}
	\end{lem}
	\begin{proof}
	  \textbf{(1)} The map $\phi$ has $\A^1$-degree $k_{\epsilon}=\langle 1 \rangle + \frac{k-1}{2}\left( \langle 1 \rangle + \langle -1 \rangle\right)$
          by \cite[Theorem 1.6]{dugger-isaksen.hopf}. Hence the claim holds with $q:=\langle 1 \rangle + \frac{k-1}{2}\left( \langle 1 \rangle - \langle -1 \rangle\right)$.
          
		\textbf{(2)} Write $k=2^r s$ with $s$ being odd. Applying \cite[Theorem~1.2]{BMP21} (see also \cite[Remark~5.5]{BMP21}) shows that the morphism 
		\[
		u\colon \A^2/(\A^2 - \{0\}) \to \A^2/(\A^2 - \{0\}),\quad u(x_1,x_2)= (x_1^2-x_2^2,x_1x_2)
		\]
		has $\A^1$-degree $3\cdot \langle 1 \rangle + \langle -1 \rangle\in GW(F)$. Similarly, the morphism
		\[
		v\colon \A^2/(\A^2 - \{0\}) \to \A^2/(\A^2 - \{0\}),\quad v(x_1,x_2)= (x_1^s,x_2^s)
		\]
		has $\A^1$-degree $s_\epsilon^2 = \langle 1 \rangle + \frac{s^2-1}{2}\left( \langle 1 \rangle + \langle -1 \rangle\right) \in GW(F)$. Then the composition
		\[
		g=\underbrace{u\circ u\circ \hdots \circ u}_r\circ v \colon \A^2/(\A^2 - \{0\}) \to \A^2/(\A^2 - \{0\})
		\]
		is given by integral homogeneous polynomials $f_1,f_2$ of degree $2^r\cdot s = k$ and has $\A^1$-degree
		\[
		(3\cdot \langle 1 \rangle + \langle -1 \rangle)^r \cdot \left(\langle 1 \rangle + \frac{s^2-1}{2}\left( \langle 1 \rangle + \langle -1 \rangle\right)\right) \in GW(F).
		\]
		For $i=1,\hdots, k$ let $f_i:=f_1$ if $i$ is odd and $f_i:=f_2$ if $i$ is even. Then the morphism
		\[
		\phi\colon \A^k/(\A^k- \{0\}) \to \A^k/(\A^k- \{0\}),\quad (x_1,x_2,\hdots,x_k)\mapsto (f_1,f_2,\hdots,f_k),	
		\]
		has $\A^1$-degree equal to the $\A^1$-degree of $\underbrace{g\wedge g \wedge \dots \wedge g}_{k/2}$ which equals to 
		\[
		\adeg(\phi)=\left(\left(3\cdot \langle 1 \rangle + \langle -1 \rangle\right)^r \cdot \left(\langle 1 \rangle + \frac{s^2-1}{2}\left( \langle 1 \rangle + \langle -1 \rangle\right)\right)\right)^{k/2} \in GW(F).
		\]
		With the definition
		\[
		q:= s^{k/2} \cdot \left(\left(3\cdot \langle 1 \rangle - \langle -1 \rangle\right)^r \cdot \left(\langle 1 \rangle + \frac{s^2-1}{2}\left( \langle 1 \rangle - \langle -1 \rangle\right)\right)\right)^{k/2} \in GW(F)
		\]
		the equality $\adeg(\phi)\cdot q = k^{3k/2}\in \Z \subseteq GW(F)$ holds, giving the statement with $M=3k/2$.
	\end{proof}
	
	\begin{lem} \label{lem:local_extension1}
		Let $V=\Spec R$ be the spectrum of a regular local ring $R$, $\xi\in V$ be the generic point and $g\colon \1^0(V) \to \1^0(\xi)$ be the restriction homomorphism. Then for even $K\in \N$ and an invertible $u \in \left(\1^0(\xi)\right)^*$ there exists $n\in \N_0$ such that $u^{K^n}=1$. In particular, $u^{K^n}\in \Im g$.			
	\end{lem}
	\begin{proof}
		\cite[Corollary~1.25]{Morel} provides an isomorphism $\1^0(\xi)\cong \GW(F(\xi))$, where $F(\xi)$ is the quotient field of $R$, so we need to show that for an invertible element $u \in \GW(F(\xi))^*$ there exists $n$ such that $u^{K^n}=1$. Since $u$ is invertible we have $\rk u=\pm 1$ and $\sgn (u)=\pm 1$ for all possible signatures $\sgn$ \cite[Definition~7.1]{baeza1978quadratic}. Changing $u$ to $u^K$ we may assume that $\rk u = 1$ and  $\sgn ( u)= 1$ for all signatures $\sgn$. Then \cite[Theorem~6.6, Theorem~7.16 and Theorem~8.8]{baeza1978quadratic} yield that $(u-1)^r=0$ and $2^s \cdot (u-1)=0$ for some $s,r\in \N$. Set $n := s +v_2(r!)$
		for the $2$-adic valuation $v_2$, then 
		\[
		u^{K^n} = (1 + (u-1))^{K^n} = 1 + \sum_{i=1}^{\min(r-1,K^n)} { K^n \choose 	i} 	\cdot  (u-1)^i = 1,
		\]
		where for the last equality we used that for $1\le i \le r-1$ one has
		\[
		v_2\left( {K^n \choose i} \right)\ge v_2(K^n) - v_2(i!) \ge n - v_2(r!) =s. \qedhere
		\]
	\end{proof}
Now, if $\sE_1$ and $\sE_2$ are vector bundles over a scheme $X$, which of the same rank, then it does not quite make sense to define the $\A^1$-degree of a map between $\Th(\sE_1)$ and $\Th(\sE_2)$. This only makes sense if $\sE_1$ and $\sE_2$ are trivial vector bundles over $X$. The next concept, while only well defined up to choices, is useful in defining a notion of degree between Thom spaces of non-trivial vector bundles, which will feature in our proof of the motivic mod $k$ Dold theorem. 
		
	\begin{defn}
		Let $\sE_1, \sE_2$ be rank $r$ vector bundles over a scheme $S$, $\zeta \in S$ be a point with the residue field $F(\zeta)$ and let $\phi\in \Hom_{\SH(S)}(\Th(\sE_1), \Th(\sE_2))$ be a morphism. For $q\in \GW(F(\zeta))$ we say that \textit{$\phi$ is of unoriented $\A^1$-degree $q$ at $\zeta$} if there exist trivializations 
		$\theta_1\colon  \Th((\sE_1)|_\zeta) \xrightarrow{\simeq} \Sigma^r_\T \1_\zeta$, $\theta_2\colon \Th((\sE_2)|_\zeta) \xrightarrow{\simeq} \Sigma^r_\T \1_\zeta$
		such that the morphism $\theta_2\circ(i^*\phi)\circ \theta_1^{-1}\in \Hom_{\SH(F(\zeta))}(\Sigma^r_\T \1_,\Sigma^r_\T \1)$ is of $\A^1$-degree $q$, where $i^*\colon \SH(S)\to \SH(F(\zeta))$ is the functor induced by the inclusion $i\colon \zeta\to S$. The abbreviation $\aged_\zeta(\phi)=q$ may be used.
	\end{defn}
	
	\begin{rem}
		Note that the above unoriented $\A^1$-degree $\aged_\zeta(\phi)$ of a morphism $\phi$ is well-defined up to multiplication by units $\GW(F(\zeta))^\times\subset\GW(F(\zeta))$, which in turn correspond to automorphisms of $\Sigma^r_\T \1_\zeta$, that is, different choices of trivializations.
	\end{rem}
	
	\begin{rem}
		Let $F$ be a field, $\sE_1, \sE_2$ be rank $r$ vector bundles over a scheme $S$ over $F$, $\zeta\in S$ be a point and $\phi\in \Hom_{\SH(S)}(\Th(\sE_1), \Th(\sE_2))$. Then in general one can not expect that there exists $q\in \GW(F)$ such that $\phi$ is of unoriented $\A^1$-degree $q$ at $\zeta$ where $q$ is viewed as an element of $\GW(F(\zeta))$ via extension of scalars. For example, take $F=\Q$ to be the field of rational numbers, $S=\A^1-\{0,1\}$, $\sE_1=\sE_2=\sO_S$. Consider morphisms
		\[
		\phi_1\colon \Th(\sO_S) \to \Th(\sO_S),\, (x,t) \mapsto (x,xt), 	\quad	\phi_2\colon \Th(\sO_S) \to \Th(\sO_S),\, (x,t) \mapsto (x,(x-1)t),
		\]
		where $x$ is the coordinate on $S$ and $t$ is the coordinate on $\sO_S$. Let $\phi= \phi_1 + \phi_2$, then the $\A^1$-degree of $i^*\phi$ for the inclusion of the generic point $i\colon \xi \to S$ is $\langle x\rangle + \langle x - 1 \rangle\in \GW(\Q(x))$. One may embed $\Q(x)\to \RR$ sending $x$ to $\pi$ and then consider the induced signature $\sgn_\pi\colon \GW(\Q(x))\to \Z$. Similarly, one may embed $\Q(x)\to \RR$ sending $x$ to $\pi/4$ and then consider the induced signature $\sgn_{\pi/4}\colon \GW(\Q(x))\to \Z$. We have 
		\[
		\sgn_\pi(\langle x\rangle + \langle x - 1 \rangle)=2,\quad \sgn_{\pi/4}(\langle x\rangle + \langle x - 1 \rangle) = 0.
		\]
		Consequently, if $v\in \GW(\Q(X))$ is an unoriented $\A^1$-degree of $\phi$ at the generic point, then it differs from $\langle x\rangle + \langle x - 1 \rangle$ by a factor from $\GW(\Q(x))^*$ whence $\sgn_{\pi}(v)=\pm 2$ and $\sgn_{\pi/4}(v)=0$. Hence $\sgn_{\pi}(v)\neq\sgn_{\pi/4}(v)$. On the other hand, for every $q\in \GW(\Q)$ one has $\sgn_{\pi}(q)=\sgn_{\pi/4}(q)=\sgn(q)$ for the standard signature $\sgn\colon \GW(\Q)\to \Z$.
	\end{rem}

	\begin{lem} \label{lem:generic_iso}
		Let $S$ be a regular scheme over a field $F$, $\sE_1, \sE_2$ be rank $r$ vector bundles over $S$. Then for a morphism $\phi\in\Hom_{\SH(S)}(\Th(\sE_1), \Th(\sE_2))$
		the following are equivalent.
		\begin{enumerate}
			\item For every generic point $\xi$ of $S$ one has $\aged_\xi(\phi)=1$.
		
			\item For every point $\zeta$ of $S$ one has $\aged_\zeta(\phi)=1$.
		
			\item The map $\phi$ is an equivalence.
		\end{enumerate}
	\end{lem}
	\begin{proof}
		Let $\Vv= \sE_2\ominus \sE_1$ and let $\alpha$ be the element corresponding to $\phi$ under the isomorphism
		\[
			\alpha\in\1^{0}(S;\Vv)=\Hom_{\SH(S)}(\1,  \Th(\sE_2)\wedge 	 \Th(\sE_1)^\vee )\cong \Hom_{\SH(S)}( \Th(\sE_1) ,  \Th(\sE_2))\ni \phi.
		\]
		
		$\mathbf{(1)\Rightarrow (2)}$
		Let $\xi\in S$ be a generic point. Then it follows from the assumption that there exists a trivialization $\theta\colon \Th(\Vv|_\xi)\xrightarrow{\simeq} \1_\xi$ such that $\Theta(\alpha|_\xi)=1$. Let $\zeta \in \overline{\{\xi\}}\subseteq S$ be a point, $V=\Spec \sO_{S,\zeta}$ and let $\upsilon\colon \Th(\Vv|_V)\xrightarrow{\simeq} \1_V$ be a trivialization. Consider the following commutative diagram. 
		\[
		\begin{tikzcd}
		  \1^0(\zeta;\Vv) \ar[d,"\Upsilon\vert_\zeta"', "\cong"] &
                  \1^0(V;\Vv) 	\ar[d,"\Upsilon"', "\cong"] \ar[r,"(-)|_\xi"]
                  \ar[l,"(-)|_\zeta"'] &
                  \1^0(\xi;\Vv) \ar[d,"\Upsilon|_\xi"', "\cong"] \\
		  \1^0(\zeta) & \1^0(V) \ar[r,"(-)|_\xi"] \ar[l,"(-)|_\zeta"'] &
                  \1^0(\xi)
                \end{tikzcd}
	          \]
		Here the vertical isomorphisms are induced by $\upsilon$ and the horizontal homomorphisms are given by restrictions. Set $u:=\Upsilon(\alpha|_V)$. Since $\Theta(\alpha|_\xi)=1$ then $u|_\xi = (\Upsilon|_\xi \circ\Theta^{-1})(1)$ is an invertible element of $\1^0(\xi)$. Then it follows from Lemma~\ref{lem:local_extension1} that $(u^N)|_\xi=1$ for some $N$ and Corollary~\ref{cor:kernel_nilpotent} yields that $u^N-1=\nu$ for a nilpotent element $\nu \in \1^0(V)$. Thus $u^N=1+\nu$ is invertible, $u$ is invertible and $u|_\zeta$ is invertible as well. Choose trivializations $\psi_1\colon \Th((\sE_1)|_{\zeta})\xrightarrow{\simeq} \Sigma^{r}_\T \1_\zeta$, $\psi_2\colon \Th((\sE_2)|_{\zeta})\xrightarrow{\simeq} \Sigma^{r}_\T \1_\zeta$ in such a way that $\psi_1\wedge (\psi_2^{-1})^\vee = u|_\zeta^{-1}\circ \Upsilon$, where $u|_\zeta$ is viewed as an automorphism of $\1_\zeta$. Then the morphism $\psi_2\circ i_\zeta^*\phi \circ \psi_1^{-1}$ is of $\A^1$-degree $1$. 
		
		$\mathbf{(2)\Rightarrow (3)}$
		The assumption yields that for every point $\zeta\in S$ with inclusion $i_\zeta\colon \zeta \to S$ the morphism $i_\zeta^*(\phi)$ is an isomorphism. The claim follows, since the collection of functors $\{i_\zeta^*\}_{\zeta \in S}$ is jointly conservative by \cite[Proposition~B.3]{norms}. 
		
		$\mathbf{(3)\Rightarrow (1)}$ This is clear.
\end{proof}
	
	\section{A motivic mod $k$ Dold theorem away from the characteristic}\label{sec:motivic-mod-k}

        The purpose of this section is to provide Theorem~\ref{thm:modkDold}, a motivic version of Adams' celebrated mod $k$ Dold theorem \cite[Theorem 1.1]{adams.j1}, which will imply the line bundle case of Theorem~\ref{thm:Adams}.
        The proof of Theorem~\ref{thm:modkDold} requires several preliminary statements, starting with the following, which should be well-known.

	\begin{lem} \label{lem:Gersten}
		Let $S$ be a regular scheme over a field $F$ and $\zeta\in S$ be a point. Set $V:=\Spec \sO_{S,\zeta}$ and let $\xi\in V$ be the generic point. Then the Cousin complex
		\[
		\1^0(V) \xrightarrow{g}		\1^0(\xi) \xrightarrow{\partial} \bigoplus_{z\in V^{(1)}} \1^{1}_{\{z\}}(\Spec \sO_{V,z})\xrightarrow{\partial}\hdots \xrightarrow{\partial} \bigoplus_{z\in V^{(d)}} \1^{d}_{\{z\}}(\Spec \sO_{V,z})
		\]
		is exact. Here $g$ is the restriction to the generic point, the direct sums are over all the points of $V$ of fixed codimension, $\partial$ are induced by the localization sequences and $d$ is the dimension of $V$.
	\end{lem}
	\begin{proof}
		Replacing $F$ with its prime subfield we may assume that $F$ is perfect. Let $f\colon \1\to \1_{\le 0}$ be the zeroth truncation of $\1=\1_F\in \SH(F)$ with respect to the homotopy $t$-structure \cite[Section~5.2]{Morel:2003}, i.e. for $p,q \in \Z$ and a local scheme $W$ essentially smooth over $F$ one has 
		\[
		(\1_{\le 0})^{p}(W; \sO_W^{\oplus q}) = \begin{cases}
			\1^{p}(W;\sO_W^{\oplus q}), & p+q =0,\\
			0, & p+q\neq 0,			
		\end{cases}
		\]
		with the isomorphism induced by $f$. Recall that by Popescu's theorem every regular local ring containing $F$ is a filtered colimit of essentially smooth local $F$-algebras \cite[Corollary~1.3]{Swan98}, whence continuity of stable motivic homotopy \cite[Proposition~C.12(4)]{HoyoisGLV} yields that the above identifications hold for an arbitrary local regular $W$ over $F$ as well.
		
		Localization sequences give rise to the strongly convergent coniveau spectral sequence
		\[
		E_1^{p,q}=\bigoplus_{z\in V^{(p)}} (\1_{\le 0})_{\{z\}}^{p+q}(\Spec \sO_{V,z}) \Rightarrow (\1_{\le 0})^{p+q}(V)
		\]
		concentrated in the strip $0\le p\le \dim V$. It follows from \cite[Proposition~4.3.10~(ii)]{DJKFundamental} (see also \cite[Theorem~C.1]{DFJK21}) that for $z\in V^{(p)}$ one has natural isomorphism
		\[
		(\1_{\le 0})_{\{z\}}^{p+q}(\Spec \sO_{V,z})\cong (\1_{\le 0})^{p+q}(z;-p\otimes \sO_z),
		\]
		whence 
		\[
			(\1_{\le 0})^{p+q}_{\{z\}}(\Spec \sO_{V,z}) = \begin{cases}
			\1^{p+q}_{\{z\}}(\Spec \sO_{V,z}), & q =0,\\
			0, & q\neq 0.			
			\end{cases}
		\]		
		Thus the coniveau spectral sequence is concentrated at the $0$-th line and coincides with the Cousin complex. The claim follows, since the target of the spectral sequence is nontrivial only when $p+q=0$ where it is given by $\1^{0}(V)$.
	\end{proof}

The next lemma is an odd counterpart to the even version in Lemma~\ref{lem:local_extension1}.

	\begin{lem} \label{lem:local_extension2}
		Let $V=\Spec R$ be the spectrum of a regular local ring $R$ containing a field $F$, $\xi\in V$ be the generic point and $g\colon \1^0(V) \to \1^0(\xi)$ be the restriction homomorphism. Then for odd $K\in \N$ and $u \in \1^0(\xi)$ if $K\cdot u \in \Im g$ then $u \in \Im g$.
	\end{lem}
	\begin{proof}
		Consider the exact sequence 
		\[
		\1^0(V) \xrightarrow{g} \1^0(\xi) \xrightarrow{\partial} \bigoplus_{z\in X^{(1)}} \1^{1}_{\{z\}}(\Spec \sO_{V,z})
		\]		
		of Lemma~\ref{lem:Gersten}. The purity property combined with \cite[Corollary~1.25, Lemma~3.10]{Morel} yields an isomorphism
		\[
		\1^{1}_{\{z\}}(\Spec \sO_{V,z}) \cong  \W(F(z))
		\]
		where $\W(F(z))$ is the Witt ring of nondegenerate symmetric bilinear forms over the residue field $F(z)$. The only possible torsion in the Witt ring is $2$-power torsion \cite[Chapter~V, Theorem~6.6]{baeza1978quadratic} whence $K\cdot \partial (u) =\partial (K\cdot u)=0$ implies $\partial (u)=0$ and the claim follows. 
	\end{proof}

	\begin{lem} \label{lm:powerkernel} 
		Let $j^*\colon S\to R$ be a homomorphism of commutative rings, and $\beta,\nu\in R$ and $K\in\N$ be such that $\nu$ is nilpotent and $\beta-\nu,\,K\cdot \nu \in j^*(S)$. Then there exists $m\in \N_0$ such that $\beta^{K^m}  \in j^*(S)$.
	\end{lem}
	\begin{proof}
		If $K=1$ then $\beta = \beta-\nu+\nu \in j^*(S)$, so we may assume $K>1$. Let $r\in \N$ be such that $\nu^{r+1}=0$ and put $m := r+\max_{p|K}v_p(r!)$ for the $p$-adic valuation $v_p$. Then 
		\[
		\beta^{K^m}= \left((\beta-\nu)+\nu\right)^{K^m} = \sum_{i=0}^{r} {K^m \choose i} \nu^i(\beta-\nu)^{K^m-i},
		\]
		here we used that $\nu^{r+1}=0$. Recall that ${K^m \choose i} = \frac{\prod_{j=0}^{j=i-1} (K^m-j)}{i!}$ whence by the choice of $m$ we have $K^i\,|\,{K^m \choose i}$ for $1\le i \le r$. Thus
		\[
		\beta^{K^m} = \sum_{i=0}^{r} a_i (K\cdot \nu)^i(\beta-\nu)^{K^m-i},\quad a_i \in \N.
		\]
		Each summand by the assumption belongs to $j^*(S)$ whence the claim.
	\end{proof}

	\color{black}
	
	\begin{lem} \label{lm:power_partial_zero} 
		Let $V=\Spec R$ be the spectrum of a regular local ring containing a field $F$, $\zeta\in V$ be the closed point, $V^o=V-\{\zeta\}$ and $\xi \in V$ be the generic point. Let $\Vv$ be a virtual vector bundle of rank $0$ over $V$ and $\theta\colon \Th(\Vv|_\xi) \xrightarrow{\simeq} \1_\xi$ be a trivialization. Let $\alpha\in \1^{0}(V;\Vv)$ and $\beta\in \1^{0}(V^o;\Vv)$
		be such that
		\[
		\alpha|_{V^o} = K\cdot \beta,\quad \Theta(\beta|_\xi)=1\in\Z\subseteq \GW(F(V))
		\]
		for some $K\in \N$. Then there exists $m\in \N_0$ such that for the localization sequence
		\[
		\1^{0}_{\{\zeta\}}(V;K^m \otimes \Vv)\to \1^{0}(V;K^m \otimes \Vv) \to \1^{0}(V^o;K^m \otimes \Vv) \xrightarrow{\partial} \1^{1}_{\{\zeta\}}(V;K^m \otimes \Vv)
		\]
		one has $\partial \left(\beta^{K^m}\right) = 0$.
	\end{lem}
	\begin{proof}
		First we introduce some notation. Since $V$ is the spectrum of a local ring, the vector bundle $\Vv$ is trivial. Let $\upsilon\colon \Th(\Vv) \xrightarrow{\simeq} \1_V$ be a trivialization. For every $m\in\N_0$ we consider the following commutative diagram.
		\[
		\begin{tikzcd}
		  \1^{0}(V;K^m \otimes \Vv) \ar[r,"j^*"] \ar[d,"\Upsilon"',"\cong"] &
                  \1^{0}(V^o;K^m \otimes \Vv) \ar[d,"\Upsilon^o_\xi"] \ar[r,"\partial"] &
                  \1^{1}_{\{\zeta\}}(V;K^m \otimes \Vv) \\
		  \1^{0}(V) \ar[r,"(-)|_\xi"] & \1^{0}(\xi) & 
		\end{tikzcd}
		\]
		Here $\Upsilon$ is induced by $\upsilon$, $\Upsilon^o_\xi=\Upsilon|_\xi\circ (-)|_{\xi}$ is the composition of the restriction to the generic point and the restricted isomorphism $\Upsilon|_{\xi}$, 
		\[
		\1^{0}(V^o;K^m \otimes \Vv) \xrightarrow{(-)|_{\xi}} \1^{0}(\xi;K^m \otimes \Vv) \xrightarrow{\Upsilon|_\xi} \1^{0}(\xi),
		\]
		and $j^*$ is the pullback. 
		
		The reasoning depends on the parity of $K$.
		
		\textbf{[$K$ is odd]} We have 
		\[
		\Upsilon(\alpha)|_\xi = {\Upsilon^o_\xi} (j^* (\alpha)) = {\Upsilon^o_\xi}( K\cdot \beta)  =  K\cdot {\Upsilon^o_\xi}(\beta)
		\]
		whence Lemma~\ref{lem:local_extension2} yields that there exists $\hat{\gamma} \in \1^0(V)$ such that $ \hat{\gamma}|_\xi = {\Upsilon^o_\xi}(\beta)$. Set $\gamma:=j^*\left(\Upsilon^{-1}(\hat{\gamma})\right)$.
		We have ${\Upsilon^o_\xi}(\gamma)={\Upsilon^o_\xi}(\beta)$ whence $\gamma|_{\xi} = \beta|_{\xi}$ and
		Corollary~\ref{cor:kernel_nilpotent} yields $\beta=\gamma+\nu$ for some nilpotent $\nu\in \1^{0}(V^o;\Vv)$. We have 
		\[
		\beta-\nu=\gamma =j^*\left(\Upsilon^{-1}(\hat{\gamma})\right)\in j^*(\1^{0}(V^o;\Vv)),\, K\cdot \nu = K\cdot \beta - K\cdot \gamma=j^*(\alpha)- K\cdot j^*\left(\Upsilon^{-1}(\hat{\gamma})\right) \in j^*(\1^{0}(V^o;\Vv)).
		\]
		Thus Lemma~\ref{lm:powerkernel} yields that there exists $m\in \N_0$ such that $\beta^{K^m}  \in j^*(\1^{0}(V^o;K^m \otimes \Vv))$ whence $\partial(\beta^{K^m}) =0$ by the exactness of the localization sequence.
		
		\textbf{[$K$ is even]}
		The assumption $\Theta(\beta|_\xi)=1$ yields that $\Upsilon^o_\xi(\beta) = (\Upsilon|_\xi \circ \Theta^{-1})(1)$ is an invertible element of $\1^{0}(\xi)$. Then Lemma~\ref{lem:local_extension1} yields that there exists $n\in \N$ and $\hat{\gamma}\in \1^0(V)$ such that $\hat{\gamma}|_\xi=\left({\Upsilon^o_\xi}(\beta)\right)^{K^n}={\Upsilon^o_\xi}(\beta^{K^n})$. Set $\widetilde{\beta}:=\beta^{K^n}$, $\widetilde{K}:=K^{K^n}$, $\widetilde{\Vv}:={K^n}\otimes \Vv$ and $\widetilde{\gamma}:=(j^*\circ \Upsilon^{-1})(\hat{\gamma})$. We have  $\widetilde{\gamma}|_{\xi} = \widetilde{\beta}|_{\xi}$ whence Corollary~\ref{cor:kernel_nilpotent} yields $\widetilde{\beta}=\widetilde{\gamma}+\widetilde{\nu}$ for some nilpotent $\widetilde{\nu}\in \1^{0}(V^o;\widetilde{\Vv})$. As in the case of $K$ being odd, we have
		\[
		\widetilde{\beta}-\widetilde{\nu},\, \widetilde{K}\cdot \widetilde{\nu} \in j^*(\1^{0}(V^o;\widetilde{\Vv})).
		\]
		Lemma~\ref{lm:powerkernel} yields that there exists $\widetilde{m}\in \N_0$ such that $\widetilde{\beta}^{\widetilde{K}^{\widetilde{m}}}  \in j^*(\1^{0}(V^o;\widetilde{K}^{\widetilde{m}} \otimes \widetilde{\Vv}))$ whence $\partial(\widetilde{\beta}^{\widetilde{K}^{\widetilde{m}}}) =0$ by the exactness of the localization sequence. The claim follows with $m=\widetilde{m}\cdot {K^n} +n$.
		\color{black}
	\end{proof}
	
	\begin{lem} \label{lm:strictify_restrictions}
		Let $\Vv$ be a virtual vector bundle of rank $0$ over a connected regular scheme $S$ over a field $F$. Let $\xi \in S$ be the generic point, $W\subseteq S$ be an open subset, $K\in \N$ and $\alpha\in \1^{0}(S;\Vv)$, $\beta\in \1^{0}(W;\Vv)$ be such that $\alpha|_\xi = K\cdot \beta|_\xi$. Then there exists $M\in \N_0$ and $\widetilde{\beta}\in \1^{0}(W;K^M\otimes \Vv)$ such that
		\begin{enumerate}
			\item $\alpha^{K^M}|_{W} = K^{K^M} \cdot \widetilde{\beta}$,
			\item $\widetilde{\beta}|_\xi = \beta^{K^M}|_\xi$.
		\end{enumerate}
	\end{lem}
	\begin{proof}
		Put $\nu=\alpha|_W - K\cdot \beta$. Then by the assumption we have $\nu|_\xi =0 $ whence Corollary~\ref{cor:kernel_nilpotent} yields that $\nu$ is nilpotent. We proceed by induction on $r$ such that $\nu^r=0$. If $r=1$ then $\alpha|_W = K\cdot \beta$ and we are done, so suppose $r>1$. Raising the equality $\alpha|_W =K\cdot \beta+\nu$ to the $K$-th power we obtain
		\[
		(\alpha^K)|_W = K^K (\beta^K + \beta^{K-1} \nu) + x\cdot \nu^2.
		\]
		Put 
		\[
		\Vv'=K\otimes \Vv,\quad \alpha'=\alpha^K\in \1^{0}(X;\Vv'),\quad \beta'= \beta^K + \beta^{K-1} \nu\in \1^{0}(W;\Vv'),\quad K' = K^K.
		\]
		Then we have $\alpha'|_W = K' \cdot \beta' + x\cdot \nu^2$ and since $\nu|_\xi=0$ then $\beta'|_\xi = \beta^K|_\xi$ and $\alpha'|_\xi = K' \cdot \beta'|_\xi$. Moreover, $(x\cdot \nu^2)^{\lceil r/2\rceil}=0$ thus we may apply induction whence there exists $M'\in\N_0$ and $\widetilde{\beta}'\in\1^{0}(W;(K')^{M'}\otimes \Vv')$ such that $(\alpha')^{(K')^{M'}}|_{W} = (K')^{(K')^{M'}} \cdot \widetilde{\beta}'$ and $\widetilde{\beta}'|_\xi = (\beta')^{(K')^{(M')}}|_\xi$. Then the claim follows with $\widetilde{\beta} = \widetilde{\beta}'$ and $M= K\cdot M'+1$.
	\end{proof}

	\begin{lem} \label{lem:mod_k_Dold}
		Let $\Vv$ be a virtual vector bundle of rank $0$ over a connected regular scheme $S$ over a field $F$, let $\xi \in S$ be the generic point and $F(\xi)$ be the field of rational functions on $S$. Suppose that there exists $\alpha\in \1^{0}(S;\Vv)$, a trivialization $\theta\colon \Th(\Vv|_\xi) \xrightarrow{\simeq}  \1_\xi$ and $k\in \Z$ such that
		\[
		\Theta(\alpha|_\xi) = k \in \Z\subseteq \GW(F(\xi)).
		\]
		Then there exists $N\in \N_0$ and $\beta\in \1^{0}(S;k^N\otimes \Vv)$ such that
		\[
		\Theta(\beta|_\xi) = 1\in \GW(F(\xi)).
		\]
	\end{lem}
	\begin{proof}
	For the proof we will inductively construct triples $(W_i,N_i,\beta_i)$, $i\in \N_0$, where
		\begin{itemize}
			\item $W_i\subseteq S$ is an open subset,
			\item $N_i\in \N_0$,
			\item $\beta_i \in \1^{0}(W_i;k^{N_i}\otimes \Vv)$,
		\end{itemize}
		satisfying
		\begin{enumerate}
			\item
			$W_{i-1}\subseteq W_{i}$ and either $W_{i-1}\subsetneq W_{i}$ or $W_{i-1}=W_{i}=S$,
			\item
			$\Theta(\beta_i|_\xi)=1\in\GW(F(\xi))$.
		\end{enumerate} 
		Since $S$ is Noetherian we have $W_n=S$ for some $n$ and the claim follows with $N=N_n,\,\beta=\beta_n$.
		
		Let $W_0\subseteq S$ be an open subset such that there exists a trivialization $\theta'\colon \Th(\Vv|_{W_0}) \xrightarrow{\simeq} \1_{W_0}$ satisfying $\theta'|_{\xi}=\theta$. Set $N_0=0$. Let $1\in \1^{0}(W_0)$ be the ring unit and denote
		\[
		\beta_0:=(\Theta')^{-1}(1)\in \1^{0}(W_0;\Vv).
		\]
		The condition $(2)$ for the triple $(W_0,N_0,\beta_0)$ clearly holds. The condition (1) is vacuous.
		
		Now suppose that we have already constructed $(W_{i},N_{i},\beta_{i})$. We may assume $W_i\neq S$; otherwise let $(W_{i+1},N_{i+1},\beta_{i+1})=(W_{i},N_{i},\beta_{i})$. By the assumption we have 
		\[
		\alpha^{k^{N_{i}}}|_\xi= \Theta^{-1} \left( k^{k^{N_{i}}}\right) = k^{k^{N_{i}}}\cdot \Theta^{-1} \left( 1\right) = k^{k^{N_{i}}} \cdot \beta_{i}|_\xi.
		\]
		Applying Lemma~\ref{lm:strictify_restrictions} to 
		\[
		\alpha^{k^{N_{i}}} \in \1^{0}(S;k^{N_i}\otimes \Vv),\quad \beta_i \in \1^{0}(W_i;k^{N_i}\otimes \Vv),\quad K=k^{k^{N_i}},
		\]
		we obtain $M\in \N$ and $\widetilde{\beta}\in \1^{0}\left(W_i;\widetilde{\Vv}\right)$, where $\widetilde{\Vv}=\left(K^M\cdot k^{N_i} \right)\otimes \Vv$, such that for 
		\[
		\widetilde{\alpha} = \left(\alpha^{k^{N_i}}\right)^{ K^M},\quad \widetilde{K}=K^{ K^M},
		\]
		we have
		\[
		\widetilde{\alpha} |_{W_i} = \widetilde{K} \cdot \widetilde{\beta}, \qquad \Theta(\widetilde{\beta}|_\xi)=1.
		\]
		Let $\zeta\in S$ be the generic point of an irreducible component of $S- W_i$ and consider the following commutative diagram.
		\[
		\begin{tikzcd}
		  \varinjlim\limits_{(\{\zeta\} \cup W_i)\subseteq U\subseteq S}\1^{0}\left(U;\widetilde{\Vv}\right) \ar[r]\ar[d]&
                  \1^{0}\left(W_i;\widetilde{\Vv}\right) \ar[r,"\partial"] \ar[d] &
                  \varinjlim\limits_{(\{\zeta\} \cup W_i)\subseteq U\subseteq S}\1^{1}_{U- W_i}\left(U;\widetilde{\Vv}\right) \ar[d,"f"] \\
		  \1^{0}\left(V;\widetilde{\Vv}\right) \ar[r,"j^*"] &
                  \1^{0}\left(V^o;\widetilde{\Vv}\right) \ar[r,"\partial"] &
                  \1^{1}_{\{\zeta\}}\left(V;\widetilde{\Vv}\right)
		\end{tikzcd}
		\]
		Here
		\begin{itemize}
			\item $V=\Spec \sO_{S,\zeta}$, $V^o=V-\{\zeta\}$,
			\item both horizontal rows are localization sequences,
			\item the vertical morphisms are induced by the canonical morphisms $V\to U$.
		\end{itemize}
		Applying Lemma~\ref{lm:power_partial_zero} to $\widetilde{\alpha}|_V$, $\widetilde{\beta}|_{V^o}$, $\widetilde{\theta}=\theta^{\oplus K^M\cdot k^{N_i}}$ and $\widetilde{K}$ we obtain $m\in \N_0$ such that
		\[
		\partial\left({\widetilde{\beta}^{\widetilde{K}^m}|_{V^o}}\right)=0.
		\]
		Note that by continuity \cite[Proposition~C.12(4)]{HoyoisGLV} morphism $f$ is an isomorphism, whence commutativity of the above diagram yields $\partial\left({\widetilde{\beta}^{\widetilde{K}^m}}\right)=0$. Since the upper row is exact then for some open subset $U\subseteq S$ satisfying $(\{\zeta\} \cup W_i)\subseteq U$ there exists $\beta_{i+1} \in \1^{0}(U;\widetilde{K}^m\otimes \widetilde{\Vv})$ such that $\beta_{i+1}|_{W_i}= {\widetilde{\beta}^{\widetilde{K}^m}}$. We set
		\[
		(W_{i+1},N_{i+1},\beta_{i+1}) = (U, k^{N_i}\cdot (m\cdot k^{M\cdot k^{N_i}}+M)+N_i,\beta_{i+1}).
		\]
		The first property is satisfied since $(\{\zeta\} \cup W_i) \subset W_{i+1}$ and $\zeta \in W_{i+1}- W_i$. The second property is satisfied since
		\[
		\beta_{i+1}|_{W_i} = {\widetilde{\beta}^{\widetilde{K}^m}}, \quad \Theta(\widetilde{\beta}|_\xi) =1. \qedhere
		\]	
	\end{proof}

	\begin{thm} \label{thm:modkDold}
	  Let $S$ be a connected regular scheme over a field $F$ with the generic point $\xi\in S$, let $\sE_1, \sE_2$ be vector bundles over $S$ of the same rank and let $k \in \Z$. Suppose that there exists a morphism $\phi\in \Hom_{\SH(S)}( \Th(\sE_1), \Th(\sE_2))$ such that $\aged_\xi(\phi)=k\in \Z\subseteq \GW(F(\xi))$.
          Then for some $N\in \N_0$ there exists an equivalence
		 $\Th(k^N\otimes \sE_1) \cong \Th(k^N\otimes \sE_2) $ in $\SH(S)$.
	\end{thm}
	\begin{proof}
		Let $\phi$ be as in the assumption, write $\Vv:= \sE_2\ominus \sE_1$ and let $\theta\colon  \Th(\Vv|_\xi) \xrightarrow{\simeq} \1_\xi$ be the respective trivialization. Let $\alpha$ be the element corresponding to $\phi$ under the isomorphism
		\[
		\alpha\in\1^{0}(S;\Vv)=\Hom_{\SH(S)}(\1_S,  \Th(\sE_2)\wedge  \Th(\sE_1)^\vee )\cong \Hom_{\SH(S)}( \Th(\sE_1) , \Th(\sE_2))\ni \phi.
		\]
		Then $\alpha$ satisfies the assumption of Lemma~\ref{lem:mod_k_Dold} whence there exists $N\in \N_0$ and $\beta\in\1^{0}(S;k^N\otimes \Vv)$ such $\Theta(\beta|_\xi)=1$. Lemma~\ref{lem:generic_iso} implies that the homomorphism $\phi'$ corresponding to $\beta$ under the isomorphism
		\[
		\phi'\in \Hom_{\SH(S)}(\Th(k^N\otimes\sE_1) , \Th(k^N\otimes\sE_2) )\cong \1^{0}(S;k^N\otimes \Vv) \ni \beta
		\]
		is an equivalence.
	\end{proof}

	\section{Brown's trick}\label{sec:browns-trick}

        The Becker-Gottlieb proof \cite{becker-gottlieb.adams} of the classical Adams conjecture uses a stable transfer map for certain fiber bundles and the deep fact that the classifying space of the monoid of self-homotopy equivalences of topological spheres is an infinite loop space.
        Brown's trick from \cite{Brown73} circumvents this fact.
        In the motivic situation it is unclear (at least to the authors) whether the classifying space of the monoid of self-$\A^1$-equivalences of motivic spheres over a field is an infinite $\P^1$-loop space; see \cite[Section 16]{norms}.
        Therefore a rather ad hoc motivic version of Brown's trick given below as Lemma~\ref{lem:Brown_trick_adhoc}, designed specifically for the transfer argument used in the proof of Theorem~\ref{thm:Adams}, comes in handy.
        Fortunately the stable transfer part has been taken care of by Hoyois and Levine. It relies on the following abstract notion of duality.
        
        \begin{defn}[{\cite{DP} and \cite{May01}}]
		Let $(\mathcal{C},\wedge,\1_{\mathcal{C}})$ be a symmetric monoidal category. An object $\sX\in \mathcal{C}$ is called \textit{strongly dualizable} if there exists $\sY\in\mathcal{C}$ and morphisms
		\[
		\mathrm{coev}_{\sX}\colon \1_{\mathcal{C}} \to {\sX}\wedge \sY,\quad
		\mathrm{ev}_{\sX}\colon \sY\wedge {\sX} \to \1_{\mathcal{C}} 
		\]
		such that the compositions
		\begin{gather*}
		{\sX}\xrightarrow{\simeq} \1_{\mathcal{C}} \wedge {\sX}\xrightarrow{\mathrm{coev}_{\sX}\wedge \id_{\sX}} {\sX}\wedge \sY\wedge {\sX} \xrightarrow{\id_{\sX}\wedge \mathrm{ev}_{\sX}} {\sX}\wedge \1_{\mathcal{C}} \xrightarrow{\simeq} {\sX},\\  \sY\xrightarrow{\simeq} \sY\wedge \1_{\mathcal{C}}\xrightarrow{\id_{\sY}\wedge \mathrm{coev}_{\sX}} \sY\wedge {\sX}\wedge \sY \xrightarrow{\mathrm{ev}_{\sX}\wedge \id_{\sY}} \1_{\mathcal{C}} \wedge \sY \xrightarrow{\simeq} \sY
		\end{gather*}
		are the respective identities. Let ${\sX}\in \mathcal{C}$ be strongly dualizable with the dual $\sY$, then the \textit{Euler characteristic} $\chi^{\mathcal{C}}({\sX})\in \End_{\mathcal{C}}(\1_{\mathcal{C}})$ is the composition
		\[
		\chi^{\mathcal{C}}({\sX})\colon \1_{\mathcal{C}} \xrightarrow{\mathrm{coev}_{\sX}} {\sX}\wedge \sY \xrightarrow{\simeq} \sY\wedge {\sX} \xrightarrow{\mathrm{ev}_{\sX}} \1_{\mathcal{C}}.
		\]
		Let $({\sX},\Delta\colon {\sX}\to {\sX}\wedge {\sX},\mu\colon {\sX}\to \1_{\mathcal{C}})$ be a coalgebra object in a symmetric monoidal category $(\mathcal{C},\wedge,\1_{\mathcal{C}})$ and suppose that ${\sX}$ is strongly dualizable with the dual $\sY$. Then the \textit{transfer} $\bglt$ is the composition
		\[
		\bglt \colon \1_{\mathcal{C}} \xrightarrow{\mathrm{coev}_{\sX}} {\sX}\wedge \sY \xrightarrow{\simeq} \sY\wedge {\sX} \xrightarrow{\id_{\sY}\wedge \Delta} \sY\wedge {\sX}\wedge {\sX} \xrightarrow{ \mathrm{ev}_{\sX}\wedge \id_{\sX}} \1_{\mathcal{C}}\wedge {\sX} \xrightarrow{\simeq} {\sX}.
		\]
	\end{defn}
		It is straightforward to check that $\mu \circ \bglt = \chi^{\mathcal{C}}({\sX})$ \cite[Remark~4.4]{May01}.
                The Euler characteristic and transfers are natural with respect to symmetric monoidal functors.
		If $\mathcal{C}=\SH$ is the classical stable homotopy category and $M$ is a compact manifold (or, more generally, a topological space of finite homotopy type) then $\chi^{\SH}({\Sigma^\infty M_+})=\chi^{\mathrm{top}}(M)$ is the classical topological Euler characteristic and the transfer is a special case of the one introduced in \cite{becker-gottlieb76}.

\begin{defn} \label{defn:bglt}
	Let $f\colon X\to S$ be a smooth morphism and $\Lambda\subseteq \Q$ be a subring. Then $f_\sharp f^*\1_S$ has a canonical coalgebra structure with $\Delta\colon f_\sharp f^*\1_S \to f_\sharp f^*\1_S\wedge f_\sharp f^*\1_S$ given by the op-lax monoidality of $f_\sharp$ and $\mu\colon f_\sharp f^*\1_S \to \1_S$ being the counit of the adjunction. Suppose that $f_\sharp f^*\1_S$ is strongly dualizable in $\SH(S)\otimes \Lambda$ with the dual $D$. The \textit{Becker-Gottlieb-Hoyois transfer} {\cite[Definition~1.5]{levine.euler}} $\bglt_f$ is the transfer for $f_\sharp f^*\1_S$, i.e. the composition
	\[
	\1_S \xrightarrow{\mathrm{coev}} f_\sharp f^*\1_S\wedge D \xrightarrow{\simeq} D\wedge f_\sharp f^*\1_S \xrightarrow{\id \wedge \Delta} D \wedge f_\sharp f^*\1_S\wedge f_\sharp f^*\1_S \xrightarrow{\mathrm{ev} \wedge \id } \1_S\wedge f_\sharp f^*\1_S \xrightarrow{\simeq} f_\sharp f^*\1_S.
	\]	
	Let $S$ be a scheme over a field of exponential characteristic $e$ and $f\colon X\to S$ be a Nisnevich locally trivial fiber bundle with the fiber $Y\in \Sm_F$, then $f_\sharp f^* \1_S$ is strongly dualizable in $\SH(S)[\frac{1}{e}]$ by the same argument as in the proof of \cite[Proposition~1.2.2]{levine.euler} combined with \cite[Theorem~3.2.1]{elden-adeel-perfection}. In particular, for such $f$ the Becker-Gottlieb-Hoyois transfer $\bglt_f\colon \1_S \to f_\sharp f^* \1_S$ is defined in $\SH(S)[\frac{1}{e}]$.
\end{defn}   	

\begin{defn}
	Let $F$ be a field. Then there are the following realization functors.

	$\mathbf{[DM]}$ There is a \textit{DM realization} functor \cite[Section~2]{Rondigs:2008}
		\[
		\Real_\DM\colon \SH(F) \to \DM(F).
		\]
		This is symmetric monoidal and satisfies $\Real_{\DM} (\Sigma_\T^\infty X_+)= M(X)$ where $M(X)$ is the Voevodsky motive of $X\in \Sm_F$ \cite{mvw}.
		
	$\mathbf{[Q_\ell]}$ 
		Let $\ell\neq \operatorname{char} F$ be a prime. Then there is an \textit{$\ell$-adic realization} contravariant functor
		\[
		\Real_{\Q_\ell}\colon \SH(F) \to D(\Q_\ell).
		\]
		It is symmetric monoidal and $H^*(\Real_{\Q_\ell}(\Sigma_\T^\infty X_+))\cong \varprojlim_n \mathrm{H}^*_{\et}(X,\Z/\ell^n\Z)\otimes_{\Z}\Q$ for $X\in\Sm_F$.
		Moreover, in view of \cite[Corollary~2.39]{Robalo}, these properties (together with the standard properties of $\ell$-adic cohomology) may be used to define the realization functor $\Real_{\Q_\ell}$. Alternatively, one can use \cite{AyoubEtale} or \cite{CDetale}. It is straightforward to see that the $\ell$-adic realization functor factors as
		\[
		\Real_{\Q_\ell}\colon \SH(F) \xrightarrow{\Real_{\DM}} \DM(F) \xrightarrow{\Real_{\Q_\ell}^{\DM}} D(\Q_\ell)
		\]
		with $\Real_{\Q_\ell}^{\DM}$ having similar properties to $\Real_{\Q_\ell}$.
		
		The functor $\Real_{\Q_\ell}$ is compatible with base change whence the following diagram commutes.
		\[
		\begin{tikzcd}
			\End_{\SH(F)}(\1_F) \ar[dd,swap,"\Real_{\Q_\ell}"] \ar[dr,pos=0.6,"g^*_{\overline{F}/F}"] &  &   & \GW(F) \ar[lll,swap,"\simeq"] \ar[dl,pos=0.6,"g^*_{\overline{F}/F}"] \ar[dd,"\rk"]\\
			 &  \End_{\SH(\overline{F})}(\1_{\overline{F}})  \ar[dl,"\Real_{\Q_\ell}"] &  \GW(\overline{F}) \ar[l,swap,"\simeq"] \ar[dr,"\rk","\simeq"'] & \\
			\End_{\D (\Q_\ell)}(\Q_\ell)&  & \Q_\ell \ar[ll,swap,"\simeq"] & \Z \ar[l] 
		\end{tikzcd}
		\]
		Here $\overline{F}$ is the algebraic closure of $F$ and $g^*_{\overline{F}/F}$ are the respective extensions of scalars. Thus if $\alpha_\Theta \in \GW(F)$ corresponds to $\Theta\in \End_{\SH(F)}(\1_F)$ then $\rk \alpha_\Theta = \Real_{\Q_\ell}(\Theta)$.
		
		\textbf{[R\'et]}
		Suppose that $F$ is formally real, i.e. $-1$ is not a sum of squares. Let $\sgn\colon \GW(F) \to \Z$ be a signature \cite[Chapter~2, Definition~4.5]{Scharlau85} corresponding to some ordering of $F$ and let $E/F$ be the respective real closure. Then there is a \textit{real \'etale realization} functor
		\[
		\Real_{\ret}\colon \SH(F) \to \SH
		\]
		given as the composition $\SH(F) \xrightarrow{g^*_{E/F}} \SH(E) \xrightarrow{R} \SH$. Here $\SH$ is the classical stable homotopy category and $R$ is the composition $\SH(E) \xrightarrow{} \SH(E)[\rho^{-1}] \cong \SH(E_{\ret}) \cong \SH$ of \cite[Theorem~35]{Bachmann18} (note that since $E$ is real closed then its small \ret-site is trivial). It is straightforward to see that if $\alpha_\Theta \in \GW(F)$ corresponds to $\Theta\in \End_{\SH(F)}(\1_F)$ then $\sgn \alpha_\Theta = \Real_{\ret}(\Theta)$ (cf. \cite[Remark~2.3]{Levine20}).
		
		If $F=\RR$ is the field of real numbers, then $\Real_{\ret}$ is canonically isomorphic \cite[Proposition~36]{Bachmann18} to the \textit{real Betti realization} functor
		\[
		\Real_{B\RR}\colon \SH(\RR) \to \SH
		\]
		which is symmetric monoidal and satisfies $\Real_{B\RR} (\Sigma^\infty_\T X_+)= \Sigma^\infty X(\RR)_+$ for $X\in\Sm_\RR$, where $X(\RR)$ is the set of real points of $X$ with the strong topology.
\end{defn}

\begin{lem} \label{lem:quadratic_invertible}
	Let $F$ be a field of exponential characteristic $e$ and $\alpha \in \GW(F)[\frac{1}{e}]$. Then $\alpha$ is a unit if and only if $\rk \alpha$ and all the possible signatures $\sgn \alpha$ (see \cite[Chapter~2, Definition~4.5]{Scharlau85}) are units in $\Z[\frac{1}{e}]$.
\end{lem}
\begin{proof}
	The only if part is clear. For the if part we treat separately the cases of $F$ being a not formally real field ($-1$ is a sum of squares) and formally real.
	
	If $F$ is not formally real then it follows from \cite[Lemma~V.7.7, Theorem~V.7.7 and Theorem~V.8.9]{baeza1978quadratic} that the fundamental ideal 
	\[
	I(F)=\ker (\GW(F) \xrightarrow{\rk} \Z)
	\]
	is the nilradical of $\GW(F)$. Then $\alpha = \pm e^m + \nu$ for some $\nu\in I(F)\otimes \Z[\frac{1}{e}]$ being nilpotent, thus $\alpha$ is a unit.
	
	Now suppose that $F$ is formally real. Then we have $e=1$ and $\rk \alpha =\pm 1$ and $\sgn \alpha =\pm 1$ for all the possible signatures. Changing $\alpha$ to $\alpha^2$ we may assume that $\rk \alpha =1$ and $\sgn \alpha =1$ for all the possible signatures. It follows from \cite[Chapter~2, Theorem~7.10 and Corollary~2.2]{Scharlau85} that $\nu = \alpha - 1 $ is nilpotent. Then $\alpha=1+\nu$ is a unit.
\end{proof}

\begin{lem} \label{lem:Brown_trick_adhoc}
	Let $F$ be a field of exponential characteristic $e$, $X=(H\backslash \GL_r)_F$ with $H=\mathrm{N}_{\GL_r} T$ being the normalizer of the standard maximal torus and let $f\colon X\to \Spec F$ be the projection. Let $\hat{\phi}\colon f_\sharp f^* \1_F \to \1_F$ be the morphism adjoint to an isomorphism $\phi\colon f^* \1_F \xrightarrow{\simeq} f^* \1_F$ in $\SH(X)[\frac{1}{e}]$. Then the composition
	\[
	\1_F \xrightarrow{\bglt_f} f_\sharp f^* \1_F \xrightarrow{\hat{\phi}} \1_F
	\]
	is an isomorphism in $\SH(F)[\frac{1}{e}]$.
\end{lem}
\begin{proof}
	The composition $\Theta = \hat{\phi}\circ \bglt_f$ corresponds to an element of $\alpha_\Theta \in \GW(F)[\frac{1}{e}]\cong \End_{\SH(F)[\frac{1}{e}]} (\1_F)$, whence Lemma~\ref{lem:quadratic_invertible} yields that it suffices to show that $\rk \alpha_\Theta$ and all the possible signatures $\sgn \alpha_\Theta$ are invertible. We will treat below the cases of the rank and of the signatures separately, but first we make the following observation. Recall that 
	\[
	\1^0(X)[\tfrac{1}{e}]=	\Hom_{\SH(X)[\frac{1}{e}]}(f^* \1_F, f^*\1_F) \cong \Hom_{\SH(F)[\frac{1}{e}]}(f_\sharp f^* \1_F, \1_F)
	\]
	and composition on the left side corresponds to the multiplication on the right induced by the diagonal morphism $\Delta\colon f_\sharp f^* \1_F \cong f_\sharp (f^* \1_F \wedge f^* \1_F) \to f_\sharp f^* \1_F \wedge f_\sharp f^* \1_F$ arising from the op-lax monoidality of $f_\sharp$. Then $\hat{\phi}$ is invertible under this product and it follows that realization functors, being symmetric monoidal, realize $\hat{\phi}$ to invertible elements of the respective rings.
	
	\textbf{[Rank]} Without loss of generality we may assume that $F$ is algebraically closed. Let $\ell\neq e$ be a prime, then $\Real_{\Q_\ell}(X)\cong \Q_\ell$ by \cite[Lemma~3.1]{Ananyevskiy22}, whence $\Real_{\Q_\ell}(\bglt_f)=\id_{\Q_\ell}$. Since the $\ell$-adic realization functor factors as
	\[
	\SH(F)[\tfrac{1}{e}] \xrightarrow{\Real_\DM} \DM(F)[\tfrac{1}{e}] \xrightarrow{\Real^\DM_{\Q_\ell}} D(\Q_\ell),
	\]
	then $\Real_{\Q_\ell}(\hat{\phi}) = \Real^\DM_{\Q_\ell} (\Real_\DM (\hat{\phi}))$ and
	\[
	\Real_\DM (\hat{\phi}) \in \Hom_{\DM(F)[\frac{1}{e}]}(\Real_\DM (f_\sharp f^* \1_F ), \Real_\DM(\1_F)) \cong CH^0(X)[\tfrac{1}{e}]=\Z[\tfrac{1}{e}].
	\]
	By the discussion above, $\Real_\DM (\hat{\phi})$ is invertible in $\Z[\tfrac{1}{e}]$. Thus 
	\[
	\rk \alpha_\Theta = \Real_{\Q_\ell}(\Theta) = \Real_{\Q_\ell}(\bglt_f)\circ \Real_{\Q_\ell}(\hat{\phi}) = \Real_{\Q_\ell}(\hat{\phi})
	\]
	is an invertible element of $\Z[\tfrac{1}{e}]$.
	
	\textbf{[Signature]} Without loss of generality we may assume that $F$ is a real closed field, in particular, $e=1$. Let $\RR^{\mathrm{alg}}=\RR\cap \overline{\Q}$ be the real closure of $\Q$ in the field of real numbers $\RR$. There is a unique isomorphism between $\RR^{\mathrm{alg}}$ and the real closure of $\Q$ in $F$ \cite[Theorem~III.2.1]{Scharlau85}, in particular, there is a canonical embedding $\RR^{\mathrm{alg}}\subseteq F$. Then we have the following commutative diagram consisting of base change functors and realizations.
	\[
	\begin{tikzcd}
		& \SH(\RR^{\mathrm{alg}}) \ar[dd,"\Real_{\ret}"] \ar[ld,swap,"g^*_{F/\RR^{\mathrm{alg}}}"] \ar[dr,"g^*_{\RR/\RR^{\mathrm{alg}}}"] &   & \\
		\SH(F) \ar[rd,"\Real_{\ret}"] & & \SH(\RR) \ar[dl,"\Real_{\ret}\cong \Real_{B\RR}"] \\
		&  \SH & 
	\end{tikzcd}
	\]
	We have $X=(H\backslash \GL_r)_{\RR^{\mathrm{alg}}}\times_{\Spec \RR^{\mathrm{alg}}} \Spec F$, whence 
	\[
	\Real_{\ret}(f_\sharp f^* \1_F) = \Real_{\ret}(f_\sharp f^* \1_{\RR^{\mathrm{alg}}}) \cong \Real_{B\RR}(f_\sharp f^* \1_\RR),
	\]
	where we denote by the same $f$ the respective projections over different base fields. Furthermore, it follows from the proof of \cite[Lemma~4.9]{Ananyevskiy22} that there is a canonical homotopy equivalence
	\[
	\Real_{B\RR}(f_\sharp f^* \1_\RR) \cong  \bigoplus_{i=1}^n \Sigma^\infty (M_i)_+
	\]
	for some connected compact real manifolds $M_i$ such that $\chi^{\mathrm{top}}(M_1)=1$ and $\chi^{\mathrm{top}}(M_i)=0$, $i\neq 1$ ($M_i$ is $\mathcal{K}_G/\mathcal{K}_N$ corresponding to $G(\RR)/N_i(\RR)$ in the notation of the proofs of \cite[Lemma~4.8 and Lemma~4.9]{Ananyevskiy22}). Then
	\[
	\sgn(\alpha_{\Theta})=\Real_{\ret}(\Theta)= \Real_{\ret} (\hat{\phi}) \circ  \Real_{\ret} (\bglt_{f}) = \sum_{i=1}^n \hat{\phi}_i \circ \bglt_i,
	\]
	where $\bglt_i\in  \Hom_{\SH}(\mathbb{S},\Sigma^\infty (M_i)_+)$ is the Becker--Gottlieb transfer for the projection $M_i\to \mathrm{pt}$ and $\hat{\phi}_i \in \Hom_{\SH}(\Sigma^\infty (M_i)_+,\mathbb{S})$ is the respective component of $\Real_{\ret}(\hat{\phi})$ and $\mathbb{S}$ is the topological sphere spectrum. Let $H\colon \mathbb{S} \to H\Z$ be the unit map for the topological Eilenberg-MacLane spectrum, then the map
	\[
	\Hom_{\SH}(\mathbb{S},\mathbb{S})\to \Hom_{\SH}(\mathbb{S},H\Z),\quad \alpha \mapsto H\circ \alpha
	\]
	is an isomorphism. We have $\Hom_{\SH}(\Sigma^\infty (M_i)_+,H\Z)\cong \Z$ with the generator given by the projection $\mu_i \colon \Sigma^\infty (M_i)_+ \to \mathbb{S}$ composed with $H\colon \mathbb{S} \to H\Z$ whence
	\[
	\sum_{i=1}^n H\circ \hat{\phi}_i \circ \bglt_i = \sum_{i=1}^n H\circ (c_i\cdot \mu_i) \circ \bglt_i
	\]
	for some $c_i\in \Z$. Recall that by the discussion in the beginning of the proof $\Real_{\ret}(\hat{\phi})$ is invertible thus $H\circ \Real_{\ret}(\hat{\phi})$ is an invertible element of the ring $\Hom_{\SH}(\bigoplus_{i=1}^n\Sigma^\infty (M_i)_+,H\Z)\cong \Z^{\times n}$ yielding that $c_i=\pm 1$. Thus
	\[
	\sgn(\alpha_{\Theta}) = \sum_{i=1}^n \hat{\phi}_i \circ \bglt_i = \sum_{i=1}^n c_i \cdot \mu_i \circ \bglt_i =\pm 1,
	\]
	where for the last equality we used
	\[
	\mu_i \circ \bglt_i = \chi^{\mathrm{top}}(M_i) = \begin{cases}
		1, & i =1,\\
		0, & i\neq 1. \qedhere
	\end{cases}
	\]
\end{proof}

\begin{rem} \label{rem:nonconnected}
	In the notation of Lemma~\ref{lem:Brown_trick_adhoc}, let $x\in X$ be a rational point with the embedding $i_x\colon \Spec F \to X$ and 
	\[
	s_x\colon \1_F\cong i_x^*f^* \1_F \xrightarrow{i_x^*\lambda_{f^*\1_F}} i_x^*f^*f_\sharp f^* \1_F \cong f_\sharp f^* \1_F
	\]
	be the corresponding morphism in $\SH(F)[\frac{1}{e}]$, where $\lambda$ is the unit of the adjunction. It is easy to see that $\hat{\phi} \circ s_x = i_x^*\phi$ is an isomorphism and one may be tempted to expect $s_x = \bglt_f$ in $\Hom_{\SH(F)[\frac{1}{e}]}(\1_F,f_\sharp f^* \1_F)= \pi_{0+(0)}(f_\sharp f^* \1_F)$ which immediately yields (if true) the claim of Lemma~\ref{lem:Brown_trick_adhoc}. Unfortunately, in general $s_x \neq \bglt_f$ since, in particular, different points $x\in X$ may give rise to different $s_x$, e.g. if $F=\RR$ and the points realize to different connected components of $X(\RR)$. This happens because $X$ is not stably $\A^1$-connected, although it is a rational variety \cite[Theorem~7.9]{BorelSpringer68} admitting an \'etale cover consisting of affine spaces (given by translates of $U_-\times U$ with $U,U_-\le \GL_r$ being the upper and lower unitriangular matrices).
\end{rem}

\begin{rem}
	Let $F$ be a field of exponential characteristic $e$, $X\in \Sm_F$ with the structure morphism $f\colon X \to \Spec F$ and $\hat{\phi}\colon f_\sharp f^* \1_F \to \sU$ be the morphism adjoint to an isomorphism $\phi \colon f^* \1_F \xrightarrow{\simeq} f^* \sU$ in $\SH(X)[\frac{1}{e}]$ for some $\sU\in \SH(F)[\frac{1}{e}]$. Suppose that $\chi^{\SH(F)}(f_\sharp f^* \1_F)=1$. Then the composition
	\[
	\1_F \xrightarrow{\bglt_f} f_\sharp f^* \1_F \xrightarrow{\hat{\phi}} \sU
	\]
	is not necessarily an isomorphism. As an example one may take $F=\CC$ the field of complex numbers, $X=X_1\sqcup X_2\sqcup X_3 $ with $X_1=X_2=\Spec \CC$ and $X_3=(\A^1_{\CC} - \{0,1\})$ and the structure morphisms $f_i\colon X_i\to \Spec \CC$, $\sU=\1_X$ and $\hat{\phi} = (\mu_1, \mu_2, -\mu_3)$ for $\mu_i\colon (f_i)_\sharp f_i^* \1_\CC \to \1_\CC$ being the respective projections (counits of the adjunctions). Here we have 
	\[
	\chi^{\SH(F)}((f_1)_\sharp f_1^* \1_\CC)=\chi^{\SH(F)}((f_2)_\sharp f_2^* \1_\CC)=1,\quad \chi^{\SH(F)}((f_3)_\sharp f_3^* \1_\CC)=-1,
	\]
	whence $\chi^{\SH(F)}(f_\sharp f^* \1_\CC)=1$, and 
	\begin{multline*}
	\hat{\phi}\circ \bglt_f= \mu_1\circ \bglt_{f_1}+ \mu_2\circ \bglt_{f_2} -\mu_3\circ \bglt_{f_3}=\\
	= \chi^{\SH(F)}((f_1)_\sharp f_1^* \1_\CC) + \chi^{\SH(F)}((f_2)_\sharp f_2^* \1_\CC) - \chi^{\SH(F)}((f_3)_\sharp f_3^* \1_\CC) = 3
	\end{multline*}
	which is not invertible in $\End_{\SH(\CC)}(\1_\CC)\cong \Z$.
\end{rem}
	
	\section{Adams' conjecture}
	\label{sec:adams-conj-away}

        With all the ingredients in place, the proof of Theorem~\ref{thm:Adams} follows an established pattern, starting with the case of line bundles.
        
	\begin{lem} \label{lem:Adams_away_line}
		Let $\sL$ be a line bundle over a regular scheme $S$ over a field $F$ and $k\in \Z$. Then for some $N\in \N_0$ one has $ \Th(k^N\otimes \sL) \cong  \Th(k^N\otimes \sL^{\otimes k})$ in $\SH(S)$.
	\end{lem}
	\begin{proof}
		Recall that $\Th(\sL^{\otimes k})\cong \Th(\sL^{\otimes -k})$ by \cite[Lemma~4.1]{Ananyevskiy20} or \cite[Prop.~2.2]{rondigs.theta}, whence we may assume $k>0$. Moreover, without loss of generality we may assume $S$ to be connected. Let $i\colon \xi\to S$ denote the inclusion of its generic point. A chosen trivialization $\theta\colon \sL|_\xi \xrightarrow{\simeq} \sO_\xi$  induces the trivialization $\theta^{\otimes k}\colon  \sL^{\otimes k}|_\xi \xrightarrow{\simeq} \sO_\xi$. We proceed separately for even and odd $k$.
		
		\textbf{[$\mathbf{k}$ is odd]} Let $\phi\colon \sL \to \sL^{\otimes k}$ be the regular morphism of schemes over $S$ given on sections by
		\[
		v\mapsto \underbrace{v\otimes v \otimes \hdots \otimes v}_k.
		\]
		Then under the trivializations $\theta,\theta^{\otimes k}$ the morphism $i^*(\Sigma_\T^\infty\Th(\phi))$ is given by the $\Sigma_\T^\infty$-suspension of
		\[
		\Th(\sO_\xi) \to \Th(\sO_\xi),\quad t\mapsto t^k,
		\]
		where $t$ is the coordinate function on $\Th(\sO_\xi)=\A^1_\xi/(\A^1_\xi - \{0\})$. Then Lemma~\ref{lem:local_degree} yields that there exists $q\in \GW(F)$ such that $q\cdot\Sigma_\T^\infty\Th(\phi)$ satisfies the assumptions of Theorem~\ref{thm:modkDold} whence the claim.
		
		\textbf{[$\mathbf{k}$ is even]} Let $q\in \GW(F)$ and $\phi\colon k\otimes \sL \to k\otimes\sL^{\otimes k}$ be the regular morphism of schemes over $S$ given on sections by $\phi$ from the second part of Lemma~\ref{lem:local_degree}. Then we may apply Theorem~\ref{thm:modkDold} to $\sE_1= k\otimes \sL$, $\sE_2= k\otimes \sL^{\otimes k}$ and the morphism $q\cdot \Sigma_\T^\infty (\phi)$ whence there exists $N'\in \N$ such that $ \Th((k^M)^{N'}\otimes (k \otimes \sE_1)) \cong  \Th((k^M)^{N'}\otimes (k \otimes \sE_2))$. The claim follows with $N=M\cdot N'+1$.
	\end{proof}	

	\begin{defn} \label{def:Adams}
		Let $k\in \N$ and $s_k\in \Z[x_1,x_2,\hdots,x_k]$ be the $k$-th Newton polynomial, that is the polynomial satisfying $t_1^k+t_2^k+\hdots + t_k^k=s_k(\sigma_1,\sigma_2,\hdots,\sigma_k)$ for the elementary symmetric polynomials $\sigma_i$. The value $\psi^k(\sE)$ of the \textit{$k$-th Adams operation} on a vector bundle $\sE$ over $S$ is the virtual vector bundle given by 
		\[
		\psi^k (\sE) = s_k(\sE, \Lambda^2\sE, \hdots, \Lambda^k \sE),
		\]
		where $\Lambda^i \sE$ is the $i$-th exterior power of $\sE$.
                Moreover, let $\psi^{-k}(\sE):=\psi^k(\sE^{\vee})$ for the dual vector bundle $\sE^\vee$ and $\psi^0(\sE)=1$, defining $\psi^k$ for all $k\in \Z$. Note that if $\sE=\sL$ is a line bundle, then $\psi^k (\sL)=\sL^{\otimes k}$ for all $k\in \Z$.	For a virtual vector bundle $\Vv=\sE_1\ominus \sE_2$ over $S$ we set 
		\[
		\psi^k (\Vv) = \psi^k (\sE_1) \ominus \psi^k (\sE_2).
		\]
		Recall \cite[Section 16.2]{norms} that $\Th$ defined on virtual vector bundles descends to
		\[
		\Th \colon K_0(S) \to \Pic (\SH(S))
		\]
		whence the usual formulas for Adams operations are applicable to $\Th\circ \psi^k$, in particular,
		\begin{itemize}
			\item 
		$\Th(\psi^k(\Vv \oplus \Vv'))\cong \Th(\psi^k(\Vv)\oplus \psi^k(\Vv'))$ 
		for $k\in \Z$ and virtual vector bundles $\Vv,\Vv'$ over $S$,
		\item
		$\Th(\psi^k(\psi^m(\Vv))) \cong \Th(\psi^{km}(\Vv))$
		for $k,m\in \Z$ and a virtual vector bundle $\Vv'$ over $S$,
		\item
		$\Th(\psi^p \sE) \cong \Th(\sE^{[p]})$ for a vector bundle $\sE$ over a scheme $S$ over a field $F$ of characteristic $p$ and $\sE^{[p]}=\mathrm{Frob}^*\sE$ with $\mathrm{Frob}\colon S\to S$ being the Frobenius morphism.
		\end{itemize}		
	\end{defn}

	\begin{defn}
		Let $\sE$ be a rank $r$ vector bundle over a scheme $S$. We say that $\sE$ \textit{admits a reduction of the structure group to $H\le \GL_r$} if there exists a (left) $H$-torsor $X\to S$ such that there exists an isomorphism of vector bundles $\sE \cong H \backslash (X\times \A^r)$ for the standard representation $H\times  \A^r \le \GL_r \times  \A^r \to \A^r$.
	\end{defn}
	
	\begin{lem} \label{lem:Adams_away_reduced} 
		Let $\sE$ be a rank $r$ vector bundle over a regular scheme $S$ over a field $F$ and $k\in \Z$. Suppose that $\sE$ admits a reduction of the structure group to $H\le \GL_r$, where $H:=N_{\GL_r} T\le \GL_r$ is the normalizer of the standard maximal torus. Then for some $N\in \N_0$ one has $ \Th(k^N\otimes \sE) \cong  \Th(k^N\otimes \psi^k \sE)$ in $\SH(S)$.
	\end{lem}
	\begin{proof}
		Let $X\to S$ be an $H$-torsor such that there exists an isomorphism $\sE \cong H \backslash (X\times \A^r)$.	Let $\widetilde{H}\le H$ be the subgroup stabilizing the decomposition $\A^1\times \A^{r-1}=\A^r$, $f\colon Y=\widetilde{H} \backslash X \to S$ be the associated degree $r$ finite \'etale morphism and $\sL=\widetilde{H} \backslash (X\times \A^1)$ be the  line bundle over $Y$ associated to the standard linear representation of $\widetilde{H}$. Recall that descent yields equivalences of categories of (equivariant) vector bundles:
		\[
		\mathcal{A}_{X/S}\colon \Vect_H(X) \xrightarrow{\simeq} \Vect(S), \quad \mathcal{A}_{X/Y}\colon \Vect_{\widetilde{H}}(X) \xrightarrow{\simeq} \Vect(Y).
		\]
		Consider the following commutative diagram.
		\[
		\begin{tikzcd}
		  \Vect(Y) \ar[d,"{R^0f_*}"]  &
                  \Vect_{\widetilde{H}}(X) \ar[l,"\simeq"',"\mathcal{A}_{X/Y}"]
                  \ar[d,"\operatorname{Ind}_{\widetilde{H}}^H"] &
                  \mathcal{R}ep(\widetilde{H})
                  \ar[d,"\operatorname{Ind}_{\widetilde{H}}^H"] \ar[l,"p_{\widetilde{H}}^*"']\\
		  \Vect(S)   &
                  \Vect_{H}(X) \ar[l,"\simeq"',"{\mathcal{A}_{X/S}}"] &
                  \mathcal{R}ep(H) \ar[l,"p_H^*"]
                \end{tikzcd}
		\]
		Here $R^0f_*$ is the direct image functor, $\mathcal{R}ep(\widetilde{H})$ and $\mathcal{R}ep(H)$ are the categories of representations over $F$,  $\operatorname{Ind}_{\widetilde{H}}^H$ is given by induction for the inclusion $\widetilde{H}\le H$, $p_H^*$ and $p_{\widetilde{H}}^*$ are induced by the projection $p\colon X\to \Spec F$. For the standard linear representation $L$ of $\widetilde{H}$ we have
		\[
		(R^0f_*)(\sL)=(R^0f_*)\circ \mathcal{A}_{X/Y} \circ p_{\widetilde{H}}^* (L) = \mathcal{A}_{X/S} \circ p_{\widetilde{H}}^* \circ \operatorname{Ind}_{\widetilde{H}}^H (L) = \sE,
		\]
		since $\operatorname{Ind}_{\widetilde{H}}^H (L)=V$ is the standard rank $r$ representation of $H$. Similarly, 
		\[
		(R^0f_*)(\sL^{\otimes k}) = \mathcal{A}_{X/S} \circ p_{\widetilde{H}}^* \circ \operatorname{Ind}_{\widetilde{H}}^H (L^{\otimes k}).
		\]
		Using the restriction functor $\mathcal{R}ep(H)\to \mathcal{R}ep(T)$ where $T\le H$ is the maximal torus it is straightforward to check that
		\[
		 \psi^k(V) = \operatorname{Ind}_{\widetilde{H}}^H (L^{\otimes k})\oplus W\ominus W
		\]
		for some representation $W$ of $H$. Here we define $\psi^k$ on $\mathcal{R}ep(H)$ using the same formulas as in Definition~\ref{def:Adams}. Summarizing the above, we obtain that
		\[
		(R^0f_*)(\sL) = \sE,\quad (R^0f_*)(\sL^{\otimes k})\oplus \sW\ominus \sW = \psi^k(\sE)
		\]
		for some vector bundle $\sW$ over $S$.
		
		Lemma~\ref{lem:Adams_away_line} yields that for some $N\in\N$ there exists an isomorphism $\theta\colon  \Th(k^N\otimes \sL)  \xrightarrow{\simeq}  \Th(k^N\otimes (\sL^{\otimes k}))$
		in $\SH(Y)$ yielding an isomorphism
		\[
		f_*(\theta)\colon f_* \Th(k^N\otimes \sL) \xrightarrow{\simeq} f_* \Th(k^N\otimes (\sL^{\otimes k})).
		\]
		in $\SH(S)$. \cite[Proposition~3.13]{norms} (see also the discussion in \cite[16.2]{norms}) yields isomorphisms
		\begin{gather*}
		f_* \Th(k^N\otimes \sL)\cong \Th(k^N\otimes (R^0f_*)(\sL)) \cong \Th(k^N\otimes \sE),\\
		f_* \Th(k^N\otimes (\sL^{\otimes k}))\cong \Th(k^N\otimes (R^0f_*)(\sL^{\otimes k})) \cong \Th(k^N\otimes \psi^k \sE),		
		\end{gather*}
		whence the claim.
	\end{proof}

	\begin{thm} \label{thm:Adams}
		Let $\sE$ be a vector bundle over a regular scheme $S$ over a field $F$ and $k\in \Z$ be an integer. Then for some $N\in \N_0$ one has $ \Th(k^N\otimes \sE) \cong \Th(k^N\otimes \psi^k \sE)$ in $\SH(S)[\tfrac{1}{e}]$, where $e$ is the exponential characteristic of $F$. 
	\end{thm}
	\begin{proof}
		Without loss of generality we may assume that $\sE$ is of constant rank $r$. Let $Y\to S$ be the Zariski locally trivial $\GL_r$-torsor associated to $\sE$ and let $X=H\backslash Y$ with $H=N_{\GL_r} T\le \GL_r$ being the normalizer of the standard maximal torus. Consider the projections $Y\xrightarrow{} X\xrightarrow{f} S$. Then $f^*\sE$ admits a canonical reduction of the structure group to $H\le \GL_r$ given by the $H$-torsor $Y\to X$, whence Lemma~\ref{lem:Adams_away_reduced} yields that there exists $N\in \N_0$ such that $\Th(k^N\otimes f^*\sE) \cong \Th(k^N\otimes \psi^k f^*\sE)$ in $\SH(X)$. Smashing this with $\Th(k^N\otimes \psi^k f^*\sE)^\vee$ we obtain an isomorphism
		$\phi\colon f^*\Th(\Vv) \xrightarrow{\simeq} f^*\1_S$, where $\Vv = k^N\otimes \sE\ominus k^N\otimes \psi^k \sE$. Let $\hat\phi \colon f_\sharp f^* \Th(\Vv) \xrightarrow{} \1_S$ be the adjoint morphism. Let
		\[
		\bglt_{f,\Vv}\colon \Th(\Vv)\cong\1_S\wedge \Th(\Vv) \xrightarrow{\bglt_f \wedge \id_{\Th(\Vv)}} f_\sharp f^* \1_S \wedge \Th(\Vv) \cong f_\sharp f^* \Th(\Vv)
		\]
		be the Becker-Gottlieb-Hoyois transfer in $\SH(F)[\frac{1}{e}]$ (see Definition~\ref{defn:bglt} and note that $f\colon X\to S$ is a Zariski locally trivial fiber bundle). We claim that the composition
		\[
		\Theta\colon \Th(\Vv) \xrightarrow{\bglt_{f,\Vv}} f_\sharp f^* \Th(\Vv) \xrightarrow{\hat\phi} \1_S
		\]
		is an isomorphism. If this is the case then smashing with $\Th(k^N\otimes \psi^k \sE)$ yields the desired isomorphism.
		
		In order to obtain the claim note that it suffices to show that for every point $\zeta\in S$ with the embedding $i=i_\zeta\colon \zeta \to S$ the morphism $i^*\Theta$ is an isomorphism, since the collection of functors $\{i^*_\zeta\}_{\zeta\in S}$ is jointly conservative \cite[Proposition~B.3]{norms}. It follows from \cite[Lemma~1.6]{levine.euler} that $i^*\Theta$ is given by the composition
		\[
		\Theta_\zeta \colon \Th(i^*\Vv) \xrightarrow{\bglt_{f_\zeta,i^*\Vv}} (f_{\zeta})_\sharp (f_{\zeta})^* \Th(i^*\Vv) \xrightarrow{\hat{\phi}_\zeta} \1_\zeta,
		\]
		where $f_\zeta \colon X_\zeta \to \zeta$ is the projection corresponding to the fiber of $f$ over $\zeta$ and $\hat{\phi}_\zeta$ is adjoint to the isomorphism $\phi_{X_\zeta} \colon \Th(f_\zeta^* \Vv) \xrightarrow{\simeq} \1_{X_\zeta}$. Note that $\Th(i^*\Vv)\cong \1_\zeta$, whence the claim follows from Lemma~\ref{lem:Brown_trick_adhoc}.
%		Now assume that $\operatorname{char} F =p>0$ (thus $e=p$), and write $k=p^n m$ with $m$ not divisible by $p$. By the above and Lemma~\ref{lem:Adams_at_char} there exists some $N\in \N$ and isomorphisms
%		\[
%		\Th(m^N\otimes \sE) \cong \Th(m^N\otimes \psi^m \sE)  \cong  \Th(m^N\otimes \psi^{p^n m} \sE)
%		\]
%		in $\SH(S)[\frac{1}{e}]$. Taking a $p^{nN}$-th $\wedge$-power provides the claim.
	\end{proof}

	\begin{rem} \label{rem:Adams_integral} 
	  If $\Sigma_\T^\infty ((N_{\GL_r} T)\backslash \GL_r)_+$ is strongly dualizable in $\SH(F)$ then the conclusion of Theorem~\ref{thm:Adams} holds without inversion of $e$ with the same proof. If $F$ is a field of characteristic $p>0$ and $k=\pm p^n$ we give below a separate proof of the Adams conjecture without inversion of $e=p$.
	\end{rem}

\begin{thm} \label{thm:Adams_at_char} 
	Let $F$ be a field of characteristic $p>0$, $S$ be a scheme over $F$, $\sE$ be a vector bundle over $S$ and $k=\pm p^n$ with $n\in \N$. Then 
	\begin{enumerate}
		\item 
		$ \Th(\sE) \cong \Th(\psi^{k} \sE)$ in $\SH(S)[\tfrac{1}{p}]$.
		\item 
		If $S$ is regular, then for some $N\in \N_0$ one has $ \Th(k^N \otimes\sE) \cong \Th(k^N\otimes \psi^{k} \sE)$ in $\SH(S)$.
	\end{enumerate}
\end{thm}
\begin{proof} 
	Without loss of generality we may assume $\sE$ to be of constant rank $r$. Recall that $\Th(\sE)\cong \Th(\sE^\vee)$ by \cite[Lemma~4.1]{Ananyevskiy20} or \cite[Prop.~2.2]{rondigs.theta}, whence we may assume $k=p^n$. Furthermore, since $\Th(\psi^p (\psi^{p^{n-1}}(\sE))) \cong  \Th( \psi^{p^n}(\sE))$ then it suffices to treat the case of $k=p$. 
	
	Consider the following commutative diagram.
	\[
	\begin{tikzcd}
		\sE \ar[drd,bend right=40] \ar[dr,pos=0.7,"\mathrm{Frob}_{/S}"] \ar[drr,bend left=30,"\mathrm{Frob}"] & & \\
		& \sE^{[p]} \ar[d] \ar[r] & \sE \ar[d] \\
		& S \ar[r,"\mathrm{Frob}"] & S
	\end{tikzcd}
	\]
	Here the square is Cartesian, the projections $\sE\to S$ are given by the structure morphism for the vector bundle and $\mathrm{Frob}$ is the Frobenius morphism. Since $\Th(\psi^p \sE)\cong \Th(\sE^{[p]})$ it sufficies to check that $\Th(\mathrm{Frob}_{/S})\colon\Th(\sE) \to \Th(\sE^{[p]})$ is an isomorphism in $\SH(S)[\frac{1}{p}]$ and that in the case of $S$ being regular there exists $N\in \N_0$ such that $\Th(k^N \otimes\sE) \cong \Th(k^N\otimes \sE^{[p]})$ in $\SH(S)$.
	
	Let $i=i_\zeta \colon \zeta \to S$ be a point, choose a trivialization $i^*\sE\cong \sO_{F(\zeta)}^{\oplus r}$ and consider the induced trivialization $i^*\sE^{[p]}\cong \sO_{F(\zeta)}^{\oplus r}$. Then under these trivializations the morphism $i^*(\Th(\mathrm{Frob}_{/S}))$ is given by the $\Sigma^\infty_\T$-suspension of the morphism
	\[
	\A^r/(\A^r - \{0\}) \to \A^r/(\A^r - \{0\}),\quad (x_1,\dots,x_r) \mapsto (x_1^p,\dots,x_r^p).
	\]
	Then it follows from \cite[Theorem 1.6]{dugger-isaksen.hopf} that under the isomorphism $\End_{\SH(F(\zeta))}\cong \GW(F)$ the morphism $i^*(\Th(\mathrm{Frob}_{/S}))$ corresponds to 
	\[
	p^r_{\epsilon} = \begin{cases}
		p^r\cdot \langle 1\rangle, & p=2,\\
		\langle 1\rangle + \frac{p^r-1}{2} (\langle 1\rangle+ \langle -1\rangle), & p>2.
	\end{cases}
	\]
	If $p>2$ we have
	\[
	\left(\langle 1\rangle + \frac{p^r-1}{2} (\langle 1\rangle+ \langle -1\rangle)\right) \cdot 	\left(\langle 1\rangle + \frac{p^r-1}{2} (\langle 1\rangle - \langle -1\rangle)\right) = p^r \cdot \langle 1\rangle.
	\]
	
	The element $p^r \cdot \langle 1\rangle$ is clearly invertible in $\GW(F)[\frac{1}{p}]$ whence $i^*(\Th(\mathrm{Frob}_{/S}))$ is an isomorphism in $\SH(F(\zeta))[\frac{1}{p}]$. The first claim of the Lemma follows, since $\{i_\zeta^*\}_{\zeta\in S}$ is a jointly conservative collection of functors \cite[Proposition~B.3]{norms}.
	
	For the second claim of the Lemma note that the above computations applied to a generic point $\xi \in S$ show that for the morphism $\phi\in \Hom_{\SH(S)}(\Th(\sE), \Th(\sE^{[p]}))$ given by
	\[
	\phi= \begin{cases}
		\Th(\mathrm{Frob}_{/S}), & p=2,\\
		\left(\langle 1\rangle + \frac{p^r-1}{2} (\langle 1\rangle - \langle -1\rangle)\right)\cdot \Th(\mathrm{Frob}_{/S}), & p>2,
	\end{cases}
	\]
	one has $\aged_\xi (\phi) = p^r \in \Z\subseteq \GW(F(\xi))$. The claim follows from Theorem~\ref{thm:modkDold}.
\end{proof}

\begin{defn}
	Let $S$ be a scheme, we say that $S$ \textit{admits a Jouanolou device} \cite[Lemma~1.5]{Jou73} if there exists a torsor under a vector bundle $X\to S$ with $X$ being affine. In particular, every scheme that is quasi-projective over an affine one, or, more generally, every scheme with an ample family of line bundles admits a Jouanolou device \cite[Section~4]{Weibel}.
\end{defn}

\begin{cor} \label{cor:singular}
	The conclusions of Theorem~\ref{thm:Adams} and Theorem~\ref{thm:Adams_at_char}~(2) also hold for a possibly singular scheme $S$ over a field $F$ assuming that $S$ admits a Jouanolou device.
\end{cor}
\begin{proof}
	We give the proof for Theorem~\ref{thm:Adams}, the case of Theorem~\ref{thm:Adams_at_char}~(2) is similar. Let $f\colon X\to S$ be a torsor under a vector bundle over $S$ with $X$ being affine. Without loss of generality we may assume that the vector bundle $\sE$ is of constant rank. Since every vector bundle over an affine scheme is a direct summand of a trivial vector bundle then for $r=\rk \sE$ and some $d\in \N$ there exists a morphism $g\colon X\to \mathrm{Gr}_F(r,d)$ such that $g^*\tau \cong f^*\sE$ for the tautological rank $r$ vector bundle $\tau$ over the Grassmannian $\mathrm{Gr}_F(r,d)$. Theorem~\ref{thm:Adams} yields that for some $N\in \N_0$ there is an isomorphism $\theta\colon \Th(k^N\otimes \tau) \xrightarrow{\simeq} \Th(k^N\otimes \psi^k \tau)$ in $\SH(\mathrm{Gr}_F(r,d))[\tfrac{1}{e}]$, where $e$ is the exponential characteristic of $F$. Then $g^*\theta\colon \Th(k^N\otimes f^*\sE) \xrightarrow{\simeq} \Th(k^N\otimes \psi^k f^*\sE)$ is an isomorphism in $\SH(X)[\tfrac{1}{e}]$ and the claim follows by \cite[Theorem~6.18(8)]{hoyois-sixops}.
\end{proof}

\appendix
\section{Torsion bounds in motivic stable homotopy groups of spheres}
\label{sec:tors-bounds-motiv}

In this section we show that the higher homotopy groups of the motivic sphere spectrum over a field are of bounded torsion (away from the exponential characteristic), an analogue of this claim was a crucial ingredient in Brown's proof of Adams conjecture \cite{Brown73}.

\begin{lem}\label{thm:witt}
	Let $F$ be a field.
	If $F$ is not formally real, the Witt group $\Witt(F)$
	is a torsion group with exponent $2s(F)$, where $s(F)$
	is the smallest number of squares which sum to $-1$;
	$s(F)$ is a power of two.
	In general, the torsion subgroup of the Witt
	group $\Witt(F)$ is the kernel of the ring homomorphism
	$\Witt(F)\to \Witt(F^{\mathrm{pyth}})$,
	where $F^{\mathrm{pyth}}$ is a Pythagorean closure of $F$.
	Its exponent is a power of two.
\end{lem}
\begin{proof}
	See \cite[Proposition~31.4 and Theorem 31.18]{ekm}.
\end{proof}

\begin{thm}\label{thm:bounded_exponent}
	Let $F$ be a field of exponential characteristic $e$ and $\1 = \1_F \in \SH(F)$ be the sphere spectrum. Then for $0<s,w\in \N$ there exists a natural number $N$ (depending only on $s$ and on $w$, and not on $F$) such that the abelian group  $\pi_{s+(w)}\1[\tfrac{1}{e}]=\Hom_{\SH(F)}(\Sigma_{\mathrm t}^w \Sigma_{\mathrm{s}}^s \1, \1)\otimes_\Z \Z[\tfrac{1}{e}]$ is $N$-torsion.
\end{thm}

\begin{proof}
	The inversion of $e$ occurs throughout and will not be
	mentioned in the notation.
	Suppose first that $e\neq 2$.
	Consider the $\eta$-arithmetic square for $\1$:
	\[
	\begin{tikzcd}
		\1 \ar{r} \ar{d} & \1[\eta^{-1}] \ar{d}\\
		\1_{\eta}^\comp \ar{r} & \1_{\eta}^\comp[\eta^{-1}]
	\end{tikzcd}
	\]
	It produces a long exact sequence of homotopy groups.
	Hence it suffices to prove the desired property for
	$\pi_{s+(w)}\1^{\wedge}_\eta$,
	$\pi_{s+(w)}\1[\eta^{-1}]$,
	and $\pi_{s+(w)}\1^{\wedge}_{\eta}[\eta^{-1}]$.
	This follows for $\pi_{s+(w)}\1[\eta^{-1}]$
	from \cite[Theorem 8.1]{bachmann-hopkins},
	using that $\pi_s\bS$ is finite for $s>0$ by Serre's thesis.
	The case $\pi_{s+(w)}\1^{\wedge}_\eta$ follows by analysis of
	the slice spectral sequence, as $\1^{\wedge}_{\eta}\simeq \slicecomp(\1)$
	by \cite[Theorem 3.50]{rso.oneline}. Here are the details.
	
	Given $q>0$, the $q$-th slice of $\1$ is a finite sum of
	motivic Eilenberg-MacLane spectra 
	\begin{equation}\label{eq:slice-sphere}
		\s_q\1 \simeq
		\bigvee_{p\geq 0} \Sigma^{2q-p,q}\mathbf{M}(\Ext_{\MU_{\ast}\MU}^{p,2q}(\MU_{\ast},\MU_{\ast}))
	\end{equation}
	with coefficients in finite abelian groups.
	Again topologists provide the finiteness of the sum and the groups, as 
	the extension groups are calculated in comodules over the Hopf algebroid
	for the cobordism spectrum $\MU$;
	these form the $E_{2}$-page of the Adams-Novikov spectral
	sequence \cite{ravenel.green}.
	Using Adams grading, the column of the $E_1$-page of the slice
	spectral sequence converging to $\pi_{s+(w)}\1^\wedge_\eta$
	is a priori infinite, but contains in every slice degree (weight)
	a finite direct sum of motivic cohomology groups with finite
	coefficients. The desired statement for $\pi_{s+(w)}\1^\wedge_\eta$
	would follow as soon as a natural number $r\geq 1$ existed
	such that the column 
	of the $E_r$-page of the slice
	spectral sequence converging to $\pi_{s+(w)}\1^\wedge_\eta$
	is finite. Indeed, the number $r=2$ works if $s$ is congruent to $1$ or $2$
	modulo $4$, by \cite[Proof of Theorem 1.1]{oro.vanishing}.
	However, if $s$ is congruent to $-1$ or $0$ modulo $4$, further conditions
	(such as $F[\sqrt{-1}]$ having finite Galois cohomological dimension at $2$)
	are required to produce such a number. Nevertheless one can deal
	with the potentially infinite columns as follows. 
	Let $\alpha_1$ denote the unique nonzero
	element in the group $\Ext_{\MU_\ast\MU}^{1,2}(\MU_\ast,\MU_{\ast})$.
	In topology, it detects the Hopf map $\eta_{\mathrm{top}}$, and over a field
	it generates $\s_1\1\simeq \Sigma^{(1)}H\Z/2$ and
	detects the Hopf map $\eta$.
	By the main result of \cite{andrews-miller}, which resolves
	Zahler's conjecture from \cite{zahler}, the
	canonical map 
	\[\Ext_{\MU_\ast\MU}^{s,t}(\MU_\ast,\MU_{\ast})\to
	\Ext_{\MU_\ast\MU}^{s,t}(\MU_\ast,\MU_{\ast})[\alpha_{1}^{-1}] \]
	is an isomorphism if  $t<\min\{6s-10,4s\}$. Moreover \cite{andrews-miller}
	provides a detailed identification of the latter algebra
	as
	\[ \Ext_{\MU_\ast\MU}^{\ast,\ast}(\MU_\ast,\MU_{\ast})[\alpha_{1}^{-1}]
	\cong \F_2[\alpha_1^{\pm 1},\alpha_3,\alpha_4]/(\alpha_4^2)\]
	with $\alpha_3$ of degree $(1,6)$ and $\alpha_4$ of degree $(1,8)$.
	One may identify this algebra with the first page of the slice spectral
	sequence for $\1[\eta^{-1}]$ by \cite[Theorem 2.3]{or.eta-periodic}.
	The resulting spectral sequence converges strongly to
	the homotopy groups of $\slicecomp(\1[\eta^{-1}])$
	by \cite[Theorem 4.6]{or.eta-periodic}. The precise 
	form of these homotopy groups has been described in
	\cite[Conjecture 4.10]{or.eta-periodic}; as \cite{bachmann-hopkins}
	resolved this conjecture in the meantime, the situation for
	$\pi_{s+(w)}\1^\wedge_\eta$ is as follows:
	There exists $m\in \bN$ (depending only on $s$ and $w$)
	and a finite filtration
	\[ 0\to \f_m\pi_{s+(w)}\1^\wedge_\eta
	\to \f_{m-1} \pi_{s+(w)}\1^\wedge_\eta \to \dotsm \to
	\pi_{s+(w)}\1^\wedge_\eta  \]
	whose associated graded pieces are subquotients of $\pi_{s+(w)}\s_j\1$, such
	that $\f_m\pi_{s+(w)}\1^\wedge_\eta$ is
	\begin{itemize}
		\item zero if $s$ is congruent to $1$ or $2$
		modulo $4$,
		\item the quotient $\Witt(F)/2^{a}$
		if $s$ is congruent to $3$ modulo $4$,
		\item the
		subgroup ${}_{2^a}\Witt(F)$ if $s$ is congruent to $0$ modulo $4$.
	\end{itemize}
	Here $a$ is a natural number depending
	only on $s$. In any case, $2^a=0$ on $\f_m\pi_{s+(w)}\1^\wedge_\eta$.
	It follows that there exists a natural number $N$, depending only on $s$ and $w$,
	such that $N=0$ on $\pi_{s+(w)}\1^\wedge_\eta$.
	
	This discussion also shows that, for every $0<s\in \bN$ there exists a $W\in \bN$
	such that multiplication with $\eta$ induces an isomorphism
	$\pi_{s+(w)}\1^\wedge_\eta\xrightarrow{\eta} \pi_{s+(w+1)}\1^\wedge_\eta$
	for all $w\geq W$. Hence the colimit computing
	$\pi_{s+(w)}\1^\wedge_\eta[\eta^{-1}]$ actually stabilizes,
	which takes care of the final member in the $\eta$-arithmetic square
	for $\1$. This completes the proof in the case $e\neq 2$.
	
	Suppose now that $e=2$. Then the exponent of $\Witt(F)$ is $2$ by Lemma~\ref{thm:witt}.
	Morel's Theorem implies that $\1[\eta^{-1}]$ -- in which $2$ is implicitly inverted --
	is a motivic ring spectrum with $\pi_{0+(0)}\1[\eta^{-1}]=0$, and hence contractible.
	The identification $\1\simeq \1^\wedge_\eta$ follows from the
	$\eta$-arithmetic square.
	The latter is equivalent to the slice completion
	$\1^\wedge_\eta\simeq \slicecomp(\1)$ by \cite[Theorem 3.50]{rso.oneline}.
	Since in $\1$ the prime $2$ is implicitly inverted, the slices of $\1$ --
	as given in \cite[Theorem 2.12]{rso.oneline} -- lead to finite columns
	of motivic cohomology groups on the $E_1$-page of the slice spectral
	sequence, by the bidegree distribution of the second page of
	the odd-primary $\BP$-based
	Adams-Novikov spectral sequence for the topological sphere
	spectrum \cite{ravenel.green}.
	As mentioned above, this already implies the desired statement for
	$\pi_{s+(w)}\slicecomp(\1)$ because the
	finitely many motivic cohomology groups
	involved in every column each have finite coefficients.  
\end{proof}

\begin{rem}\label{rem:hopkins-morel}
	The conjectured Hopkins-Morel isomorphism in characteristic $p$ provides a computation of the $p$-local motivic stable stems; see e.g. \cite[Section~1.2]{Bachmann22}. Hence Theorem~\ref{thm:bounded_exponent} should hold without inverting the exponential characteristic. 
\end{rem}		
%
%\begin{cor} \label{cor:bounded_torsion}
%	Let $V=\Spec R$ be the spectrum of a regular local ring $R$ containing a field $F$, $\zeta\in V$ be the closed point and let $\Vv$ be a virtual vector bundle of rank $0$ over $V$. Suppose that $d=\dim V \ge 2$ and that $K\in \N$ is invertible in $R$. Then the $K$-power torsion subgroup of $\bigoplus_{n\ge 0} \1^{1}_{\{\zeta\}}(V;n\otimes \Vv)$ is of finite exponent.
%\end{cor}
%\begin{proof}
%	Since $V$ is the spectrum of a local ring then there exists a trivialization $\theta \colon \Th(\Vv) \xrightarrow{\simeq} \1_V$ giving rise to the isomorphisms 
%	\[
%	\Theta \colon\1^{1}_{\{\zeta\}}(V;n\otimes \Vv)\xrightarrow{\simeq}\1^{1}_{\{\zeta\}}(V).
%	\]
%	Furthermore, the purity property yields isomorphism
%	\[
%	\1^{1}_{\{\zeta\}}(V) \cong \1^{1}(\zeta; -d\otimes \sO_{\zeta}) = \pi_{d-1+(d)} \1_{\zeta}.
%	\]
%	Then the claim follows from Theorem~\ref{thm:bounded_exponent}.
%\end{proof}
	\let\mathbb=\mathbf

\end{document}